\documentclass[11pt,x11names,dvipsnames,svgnames]{amsart}
\usepackage{amsfonts,amssymb,amscd,amstext,mathrsfs,mathtools}
\usepackage{hyperref}
\usepackage{verbatim}
\usepackage{comment}
\usepackage[T1]{fontenc} 

\usepackage{graphicx}
\usepackage[up,bf]{caption} 
\usepackage{subcaption}

\usepackage{newtxtext,newtxmath}
\usepackage{enumerate}

\numberwithin{equation}{section}
\numberwithin{figure}{section}

\newtheorem{theorem}{Theorem}[section]

\newtheorem{lemma}[theorem]{Lemma} 
\newtheorem{corollary}[theorem]{Corollary} 

\theoremstyle{definition} 
\newtheorem{definition}[theorem]{Definition} 
\newtheorem{remark}[theorem]{Remark} 
\newtheorem{question}[theorem]{Question}

\newcommand\bA{\mathbf{A}}

\newcommand\Jcal{\mathcal{J}}

\newcommand\bu{\mathbf{u}} 
\newcommand\bv{\mathbf{v}}

\newcommand\bw{\mathbf{w}}

\newcommand\Ascr{\mathscr{A}} 
 
\newcommand\Cscr{\mathscr{C}}

\newcommand\B{\mathbb{B}} 
\newcommand\C{\mathbb{C}} 
\newcommand\D{\mathbb D} 
\newcommand\cD{\overline{\mathbb D}} 
\newcommand\CP{\mathbb{CP}}

\newcommand\N{\mathbb{N}} 
 
\newcommand\R{\mathbb{R}}

\newcommand\T{\mathbb{T}} 
 
\newcommand\Z{\mathbb{Z}}

\newcommand\E{\mathrm{e}}

\newcommand\wt{\widetilde} 
 
\newcommand\di{\partial} 
\newcommand\dibar{\overline\partial}

\newcommand\dist{\mathrm{dist}} 
 
\newcommand\length{\mathrm{length}}

\newcommand\Id{\mathrm{Id}}

\def\Ell1{\mathrm{Ell_1}} 
\def\CEll1{\mathrm{CEll_1}}

\newcommand\Flux{\mathrm{Flux}} 
\newcommand\GCMI{\mathrm{GCMI}} 
\newcommand\CMI{\mathrm{CMI}}

\begin{document} 

\title{The Calabi--Yau and Sullivan conjectures for minimal surfaces \\
with symmetries}

\author{Antonio Alarc\'on and Franc Forstneri{\v c}}

\address{Antonio Alarc\'on, Departamento de Geometr\'{\i}a y Topolog\'{\i}a e Instituto de Matem\'aticas (IMAG), Universidad de Granada, Campus de Fuentenueva s/n, E--18071 Granada, Spain.}

\email{alarcon@ugr.es}

\address{Franc Forstneri{\v c}, Faculty of Mathematics and Physics, University of Ljubljana, Jadranska 19, 1000 Ljubljana, Slovenia}

\address{Franc Forstneri{\v c}, Institute of Mathematics, Physics, and Mechanics, Jadranska 19, 1000 Ljubljana, Slovenia}

\email{franc.forstneric@fmf.uni-lj.si}


\date{22 October 2026}


%
%
\begin{abstract}
Let $G$ be a finite group acting on a connected open Riemann surface 
$X$ by holomorphic automorphisms and acting on a Euclidean space 
$\mathbb R^n$, $n\ge 3$, by orthogonal transformations. 
In a previous paper, the second author gave a necessary and 
sufficient condition for the existence of a $G$-equivariant conformal minimal 
immersion $F:X\to\mathbb R^n$, and he asked whether such $F$ can be 
chosen to be a proper map. Assuming that this condition holds, 
we give an affirmative answer when $X$ is a bordered 
Riemann surface and $G$ acts freely on the boundary of $X$. 
This contributes to the Sullivan conjecture concerning the existence of proper hyperbolic minimal surfaces in $\R^3$ of finite topology.
Under the same conditions, we also construct 
a $G$-equivariant proper conformal minimal immersion 
of $X$ to the unit ball of $\R^n$ which is complete in the induced
metric, thereby opening a new direction in the Calabi--Yau theory 
regarding the study of complete bounded minimal surfaces. 
Analogous results hold for null curves in $\C^n$.
\end{abstract}

\maketitle

%
%
\section{Introduction}\label{sec:intro}

The classical Calabi-Yau and Sullivan conjectures concern the existence of complete bounded minimal surfaces in $\R^3$, and of proper hyperbolic minimal surfaces in $\R^3$ of finite topology, respectively. Since the pioneering positive results of Jorge-Xavier \cite{JorgeXavier1980AM} and Nadirashvili \cite{Nadirashvili1996IM} concerning the former, and of Morales \cite{Morales2003GAFA} regarding the latter, a primary topic of interest has been to understand the topological, geometric, and conformal properties of such surfaces; we refer to \cite[Sec.\ 3.10 and Ch.\ 7]{AlarconForstnericLopez2021} for surveys on these questions. Concerning possible conformal restrictions, Alarc\'on and L\'opez proved that every open Riemann surface is the underlying complex structure of a proper minimal surface in $\R^3$ \cite{AlarconLopez2012JDG}, while the authors showed that every bordered Riemann surface is the complex structure of a complete bounded minimal surface in $\R^n$
\cite{AlarconForstneric2015MA}. 

In this paper, we establish analogues of these existence results for minimal surfaces 
with a prescribed finite group of symmetries that are normalised by bordered 
Riemann surfaces. Our main new results are summarised in 
Theorems \ref{th:proper}, \ref{th:complete}, and \ref{th:complete-proper}.
To the best of our knowledge, these are the first 
general existence results regarding the Calabi-Yau and Sullivan conjectures 
for minimal surfaces with symmetries in Euclidean spaces.

Assume that $X$ is an open Riemann surface and $G$ is a finite group 
acting effectively on $X$ by holomorphic automorphisms. 
For any point $x\in X$, the stabiliser 
\begin{equation}\label{eq:stabiliser}	
	G_x=\{g\in G:gx=x\}
\end{equation}
is a cyclic subgroup of $G$ 
(see Miranda \cite[Proposition 3.1, p.\ 76]{Miranda1995}), and 
\begin{equation}\label{eq:X0}
	X_0 = \{x\in X: G_x\ne \{1\}\} 
\end{equation}
is a closed, discrete, $G$-invariant subset of $X$
(see \cite[Proposition 3.3, p.\ 76]{Miranda1995}).
The stabiliser $G_x$ acts on the tangent space $T_xX\cong\C$ as a cyclic 
group of rotations by 
\[
	gv=dg(v),\quad g\in G_x,\ v\in T_xX.
\]
Furthermore, there is a local holomorphic coordinate $z$ on $X$, 
with the point $x$ corresponding to $z=0$, in which 
a generator $g$ of $G_x$ is the rotation 
\begin{equation}\label{eq:rotation}
	g z = \E^{\imath \phi}z,\quad \phi=2\pi/|G_x|,\ \imath=\sqrt{-1}
\end{equation}
(see \cite[Corollary 3.5, p. 93]{Miranda1995}).
Assume that the group $G$ also acts on the Euclidean space $\R^n$. 
A map $F:X\to \R^n$ is said to be {\em $G$-equivariant} if 
\begin{equation}\label{eq:equivariant}
	F(gx) = gF(x)\ \ \text{for all $x\in X$ and $g\in G$}.
\end{equation}
The problem of existence, behaviour, and classification of $G$-equivariant 
maps has been studied in many geometries; see the discussion and 
references in 
\cite{DAngelo2021,KutzschebauchLarussonSchwarz2022},
among other sources. Recently, the second 
author gave the following necessary and sufficient condition 
for the existence of a $G$-equivariant conformal minimal immersion 
$X\to \R^n$, $n\ge 3$; see \cite[Theorem 1.1]{Forstneric2024PLMS}
and Remark \ref{rem:cond-a} below.

%
%
\begin{theorem}[Corrected form of Theorem 1.1 in \cite{Forstneric2024PLMS}] 
\label{th:FF}
Let $G$ be a finite group acting effectively on a connected open 
Riemann surface $X$ by holomorphic automorphisms and acting on 
the Euclidean space $\R^n$, $n\ge 3$, by orthogonal transformations. 
The following conditions are equivalent:
\begin{enumerate}[\rm (a)] 
\item 
For every point $x\in X$ with nontrivial stabiliser $G_x$ there exists a 
$G_x$-equivariant conformal linear embedding $T_x X \to \R^n$. 
\item 
There exists a $G$-equivariant conformal minimal immersion 
$F:X\to \R^n$. 
\end{enumerate}
If (a) holds then $F$ in (b) can be chosen such that 
$F(X)$ is not contained in any affine hyperplane of $\R^n$.
\end{theorem}

Condition (a) is equivalent to the following. 
Let $x\in X$ be a point with a nontrivial stabiliser group $G_x$,
and let $g$ be a generator of $G_x$ such that $dg_x:T_x X\to T_xX$ 
is the rotation $R_\phi$ through the angle $\phi=2\pi/|G_x|$ in the
positive direction.  
Then there exists a two-dimensional real linear subspace 
$\Lambda\subset \R^n$ on which $g$ also acts as the rotation $R_\phi$
in one of the two orientations of $\Lambda$. 

A more precise version of this result is given by 
\cite[Theorem 5.1]{Forstneric2024PLMS}; see Theorem \ref{th:FFbis}.
Recall that a conformal minimal immersion \(F:X\to\R^n\) is called 
{\em nondegenerate} if \(F(X)\) is not contained in an affine hyperplane of 
$\R^n$, and {\em nonflat} if its image is not contained in an affine 
two-plane. Thus, when $n>3$, nondegeneracy is strictly stronger than 
nonflatness.

%
%
\begin{remark}\label{rem:cond-a}
Condition (a) in Theorem \ref{th:FF} is stronger than the one in 
\cite[Theorem 1.1]{Forstneric2024PLMS}, which reads
\begin{enumerate}[\rm (a')]
\item
For every nontrivial stabiliser $G_x$, $x\in X$, there is a two-dimensional linear subspace $\Lambda_x\subset\R^n$ on which $G_x$ acts effectively as a cyclic group of rotations.
\end{enumerate}
This condition does not suffice for the existence of a $G_x$-equivariant 
conformal linear isomorphism 
$T_x X \stackrel{\cong}{\to} \Lambda_x \subset \R^n$, 
which is used in \cite{Forstneric2024PLMS}. The point
is that a generator of $G_x$ must act as a rotation by the same angle on
both $T_xX$ and $\Lambda_x$. Hence, condition (a') must be replaced by 
condition (a) in Theorem \ref{th:FF} in all existence theorems in 
\cite{Forstneric2024PLMS}. On the other hand, condition (a') 
ensures the existence of $G$-equivariant conformal minimal surfaces
$F:X\to\R^n$ which are branched at those points of $X_0$ 
\eqref{eq:X0} where condition (a) fails; see Theorem \ref{th:branched}.
\end{remark}

Assuming that condition (a) in Theorem \ref{th:FF} holds, several 
questions about possible additional properties of a conformal minimal 
immersion $F$ in part (b) of the theorem were posed in 
\cite{Forstneric2024PLMS}. 
It was asked in particular whether $F$ can be chosen to be a
{\em proper map} to $\R^n$; see \cite[Problem 1.13]{Forstneric2024PLMS}.  
We give an affirmative answer in the important special case when $X$ is a 
bordered Riemann surface and $G$ acts freely on its boundary $bX$.
(Throughout the paper, a bordered Riemann surface means the interior 
\(X\) of a connected compact domain \(\overline X=X\cup bX\) with nonempty 
smooth boundary $bX$ in a Riemann surface $R$.) 
Since every such surface is of finite 
topology and hyperbolic conformal type, the following result pertains to the 
study of the aforementioned Sullivan conjecture.

%
%
\begin{theorem} \label{th:proper}
Let $X$ be a bordered Riemann surface and $G$ be 
a finite group acting effectively on $X$ by holomorphic 
automorphisms such that no element $g\in G\setminus \{1\}$ 
has any fixed points on $bX$, i.e.\ $G$ acts freely on $bX$. 
Assume that $G$ also acts on $\R^n$, $n\ge 3$, by 
orthogonal transformations. If condition (a) in Theorem \ref{th:FF}
holds then there exists a proper, nondegenerate, $G$-equivariant 
conformal minimal immersion $F:X\to \R^n$.
\end{theorem}

%
%
\begin{remark}\label{rem:XG}
We may assume without loss of generality that the bordered Riemann surface
$X$ in Theorem \ref{th:proper} is a relatively compact $G$-invariant domain 
with real analytic boundary in an open Riemann surface $R$ and that 
$G$ acts on $R$ by holomorphic automorphisms.
Indeed, by \cite[Theorem 8]{Stout1965TAMS} the surface
$X$ is biholomorphic to 
the complement $X=X^*\setminus \bigcup_{j=1}^m \Delta_j$
of finitely many closed pairwise disjoint discs $\Delta_j$ with
real analytic boundaries in a compact Riemann surface $X^*$.
By the Schwarz reflection principle, every holomorphic 
automorphism of $X$ extends to an injective holomorphic map
on a neighbourhood of $\overline X$ in $X^*$.
Since the group $G$ is finite, there is a 
neighbourhood $U$ of $\overline X$ on which all these maps 
are defined, and they satisfy the group law.
The connected component $R$ of the $G$-invariant domain 
$\bigcap_{g\in G} gU$ containing $\overline X$ is then an open Riemann 
surface on which $G$ acts by holomorphic automorphisms.
More generally, Maskit \cite{Maskit1968} proved that every open 
Riemann surface $X$ of finite genus embeds in a compact Riemann surface 
$X^*$ of the same genus such that every holomorphic automorphism 
of $X$ extends to a holomorphic automorphism of $X^*$. 
\end{remark}

Theorem \ref{th:proper} is proved by combining the technique
from \cite{Forstneric2024PLMS}, which was used to prove 
Theorem \ref{th:FF} (see the exposition in Section \ref{sec:prelim}), 
with the method of constructing proper conformal minimal immersions 
from a bordered Riemann surface $X$ to any convex domain 
in $\R^n$ (see \cite[Theorem 1.4 and Corollary 1.5]{AlarconDrinovecForstnericLopez2015PLMS}) by using the technique for 
exposing boundary points and the Riemann--Hilbert modification method 
for conformal minimal surfaces; see Subsections  \ref{sec:exposing} and 
\ref{sec:RH}. In the case at hand, the target convex domain 
is $\R^n$ itself, which simplifies the construction.

Another method for constructing proper minimal immersions
of an arbitrary open Riemann surface to $\R^n$ for any $n\ge 3$,
relying on the Runge-Mergelyan approximation theorems for conformal 
minimal immersions, was first introduced in \cite{AlarconLopez2012JDG} 
for $n=3$ and developed in full generality in 
\cite{AlarconForstnericLopez2016MZ}; see also
\cite[Theorem 3.10.3]{AlarconForstnericLopez2021}.
However, the induction step in this method requires one
to fix some of the components of the map under construction
to be harmonic functions on $X$, and we were unable 
to adapt it to the equivariant situation.
Nevertheless, combining the proofs of Theorems \ref{th:FF}
and \ref{th:proper} implies the following corollary.

%
%
\begin{corollary}\label{cor:almostproper}
Let $X$ and $G$ be as in Theorem \ref{th:FF}, satisfying condition (a).
Then there exists an almost proper $G$-equivariant conformal 
minimal immersion $X\to \R^n$.
\end{corollary}

Recall that a continuous map $F:X\to Y$ of topological spaces
is called almost proper if for every compact set $K\subset Y$,
every connected component of the preimage $F^{-1}(K)$ is compact.

%
%
Another question related to the Sullivan conjecture that has been a focus of 
interest is that by Schoen and Yau, asking whether there are hyperbolic minimal 
surfaces in $\R^3$ with a proper projection into a plane (see 
\cite[p.\ 18]{SchoenYau1997} and \cite[Sec.\ 3.10]{AlarconForstnericLopez2021}). 
This problem was settled in \cite{AlarconLopez2012JDG}, where it was shown 
that projecting properly into a plane does not impose any restriction on the 
underlying complex structure of a minimal surface in 
$\R^3$. The following remains as an open question.

\begin{question}
Let $X$ and $G$ be as in Theorem \ref{th:FF}, satisfying condition (a). 
Does there exist a $G$-equivariant conformal minimal immersion 
$X\to\R^3$ having a proper 
projection into a plane?
\end{question}

Let $ds^2$ denote the Euclidean metric on $\R^n$.
Let us call a $\Cscr^1$ immersion $F:X\to \R^n$ bounded if its image $F(X)$
is contained in a ball, and complete if the induced Riemannian metric 
$F^*ds^2$ on $X$ is a complete metric. 
With $X$ and $G$ as in Theorem \ref{th:proper},
it is natural to ask whether there exists a $G$-equivariant 
conformal minimal immersion $F:X\to \R^n$ which is 
bounded and complete (see \cite[Problem 1.12]{Forstneric2024PLMS}).
We provide the following affirmative answer. 

%
%
%
\begin{theorem} \label{th:complete}
Let $X$ and $G$ be as in Theorem \ref{th:proper}.
If condition (a) in Theorem \ref{th:FF} holds then there exists a 
continuous $G$-equivariant map $F: \overline X\to\R^n$ 
whose restriction $F|_{X}: X\to\R^n$ is a 
complete, nondegenerate conformal minimal immersion with a 
bounded image. Furthermore, any nonflat $G$-equivariant conformal
minimal immersion $F_0:\overline X\to\R^n$ can be 
approximated uniformly on $\overline X$ by maps 
$F: \overline X\to\R^n$ as above. 
\end{theorem}

Without the $G$-equivariance condition, Theorem \ref{th:complete} is 
a special case of 
\cite[Theorem 1.1]{AlarconDrinovecForstnericLopez2015PLMS}. 
Combining the methods of proof of Theorems \ref{th:proper}
and \ref{th:complete} we obtain the following.

%
%
%
\begin{theorem} \label{th:complete-proper}
Assume that $X$ and $G$ are as in Theorem \ref{th:proper}
and condition (a) in Theorem \ref{th:FF} holds.
Then there exists a continuous $G$-equivariant 
map $F: \overline X\to \overline \B^n$ 
to the closed unit ball of $\R^n$ such that the restriction
$F|_X:X\to \B^n$ is a complete, proper, nondegenerate $G$-equivariant 
conformal minimal immersion. 
\end{theorem}

Theorems \ref{th:complete} and \ref{th:complete-proper}
contribute to the circle of results on the 
famous Calabi--Yau problem for minimal surfaces; see 
\cite[Chap.\ 7]{AlarconForstnericLopez2021} for a discussion of this topic. 
Their proof provides a general construction of complete bounded minimal 
surfaces in $\R^n$ with a prescribed conformal structure and a given group 
of symmetries, thereby opening a new direction in the theory. 
Without $G$-equivariance, Theorem \ref{th:complete-proper} 
is a special case of 
\cite[Theorem 1.2]{AlarconDrinovecForstnericLopez2015PLMS}
which gives the analogous result with the ball replaced by an arbitrary
strongly convex domain with $\Cscr^2$ boundary in $\R^n$, $n\ge 3$. 
See also \cite{AlarconDrinovecForstnericLopez2019TAMS} and the 
exposition in \cite[Secs.\ 8.3--8.5]{AlarconForstnericLopez2021}
where similar results were proved for minimal surfaces in mean 
convex domains in $\R^3$ and $(n-2)$-convex domains in 
$\R^n$ for $n\ge 4$.

%
%
\begin{remark}\label{rem:flux}
Let $H_1(X,\Z)$ denote the first homology group of $X$.
The flux homomorphism $\Flux_F:H_1(X,\Z)\to\R^n$ 
of a conformal minimal immersion $F:X\to\R^n$ is defined by 
\begin{equation}\label{eq:flux}
	\Flux_F([C]) = \oint_C d^c F,\quad [C]\in H_1(X,\Z),
\end{equation}
where $d^c$ is the conjugate differential on $X$.
(See \cite[Definition 2.3.2]{AlarconForstnericLopez2021}.) 
An action of a group $G$ on $X$ induces an action of $G$ 
on $H_1(X,\Z)$. If a conformal minimal immersion $F:X\to\R^n$
is $G$-equivariant then clearly so is its flux homomorphism.
In Theorems \ref{th:proper}, \ref{th:complete}, and \ref{th:complete-proper} 
one can also control the flux homomorphism of a $G$-equivariant 
conformal minimal immersion to be any given $G$-equivariant 
homomorphism $H_1(X,\Z)\to \R^n$. This will be clear from the proofs;  
see also \cite[Remark 5.2]{Forstneric2024PLMS}.
\end{remark}

The analogues of Theorems \ref{th:proper} and \ref{th:complete} 
also hold for {\em holomorphic null curves}, that is, 
holomorphic immersions $H=(H_1,\ldots,H_n):X\to \C^n$, $n\ge 3$, 
whose differential $dH$ satisfies the nullity condition 
\[
	\sum_{i=1}^n (dH_i)^2=0.
\]
Let $O(n,\C)=\{A\in GL(n,\C):AA^t =I\}$ denote the complex
orthogonal group on $\C^n$. The null quadric $\bA\subset \C^n$ 
\eqref{eq:bA} is $O(n,\C)$-invariant, and the group $O(n,\C)$
acts transitively on $\bA_*=\bA\setminus \{0\}$. 
For reference, we recall the following analogue of Theorem \ref{th:FF}; 
see \cite[Theorem 5.3]{Forstneric2024PLMS}. 
The special case with $G$ the trivial group is given by 
\cite[Theorem 3.6.1]{AlarconForstnericLopez2021}. 

%
%
\begin{theorem}[Theorem 5.3 in \cite{Forstneric2024PLMS}] \label{th:FFnull}
Let $G$ be a finite group acting effectively on a connected open 
Riemann surface $X$ by holomorphic automorphisms, and acting on 
$\C^n$ $(n\ge 3)$ by complex orthogonal transformations in $O(n,\C)$. 
Assume that for every point $x\in X$ with nontrivial stabiliser $G_x$ 
there is a $G_x$-invariant complex line $L_x\subset \bA \subset\C^n$ 
and a $G_x$-equivariant $\C$-linear isomorphism $T_xX\to L_x$.
Then there exists a nondegenerate $G$-equivariant holomorphic 
null immersion $H:X\to \C^n$.
\end{theorem}

Part (a) of the following result is an analogue of Theorem \ref{th:proper} 
for $G$-equivariant holomorphic null curves. Part (b) is an analogue
of a special case of Theorem \ref{th:complete}.

%
%
\begin{theorem} \label{th:NC}
Let $X$ and $G$ be as in Theorem \ref{th:proper}.
Assume that $G$ also acts on $\C^n$, $n\ge 3$, by 
elements of $O(n,\C)$ and the condition on the stabilisers 
in Theorem \ref{th:FFnull} is satisfied. Then the following hold.
\begin{enumerate}[\rm (a)]
\item There exists a proper $G$-equivariant 
holomorphic null immersion $H:X\to \C^n$ whose image $H(X)$ is not 
contained in any affine hyperplane of $\C^n$.
\item 
There exists a continuous $G$-equivariant map $H: \overline X\to\C^n$ 
whose restriction $H|_{X}: X\to\C^n$ is a bounded 
complete holomorphic null immersion such that $H(X)$ is not contained
in any affine hyperplane of $\C^n$.
\end{enumerate}
\end{theorem}

Part (b) contributes to the {\em Yang problem} concerning 
the existence and properties of complete bounded
complex submanifolds of complex Euclidean spaces.
See \cite{Alarcon2024Yang} for a recent survey.

%
%
The paper is organised as follows. 
In Section \ref{sec:prelim} we outline the proof of Theorem \ref{th:FF}
given in \cite{Forstneric2024PLMS}. 
The same general scheme but with suitable technical additions is 
then used to prove Theorems \ref{th:proper}, \ref{th:complete}, and  
\ref{th:complete-proper} in Section \ref{sec:proof} via recursive applications 
of the main properness and completeness lemmas in Section 
\ref{sec:lemmas}. The same methods yield the proof of Theorem 
\ref{th:NC} also given in Section \ref{sec:proof}. The necessary additions 
to establish the main lemmas include $G$-equivariant versions 
of the technique of exposing boundary points of bordered Riemann surfaces
(see Theorem \ref{th:exposing} in Subsection \ref{sec:exposing}), and of the 
Riemann--Hilbert deformation method for conformal minimal surfaces 
and holomorphic null immersions (see Theorem \ref{th:RH} and Corollary 
\ref{cor:RH-null} in Subsection \ref{sec:RH}). 
This classical complex analytic technique for solving certain boundary 
value problems for holomorphic maps 
was first introduced to the subject of conformal 
minimal surfaces by the authors in \cite{AlarconForstneric2015MA}, 
where it was shown that every bordered Riemann surface admits a 
complete bounded conformal minimal immersion in $\R^3$. 
In the subsequent paper \cite{AlarconDrinovecForstnericLopez2015PLMS}, 
the Riemann--Hilbert technique was extended to minimal surfaces in 
$\R^n$ for arbitrary $n\ge 3$, and it was used to give 
complete bordered minimal surfaces with boundaries consisting 
of Jordan curves. (The boundary curves are necessarily nonrectifiable 
at every point.) Soon thereafter, the technique was improved for $n=3$ and
applied in the construction of complete, proper conformal minimal 
immersions of bordered Riemann surfaces to strongly minimally 
convex domains in $\R^3$ (in particular, to strongly convex domains); 
see \cite{AlarconDrinovecForstnericLopez2019TAMS}.
A detailed presentation of this method, and 
of its applications to the Calabi--Yau problem, can be found in 
\cite[Chaps.\ 6--7]{AlarconForstnericLopez2021}.

%
%
\section{Outline of proof of Theorem \ref{th:FF}}\label{sec:prelim} 

In this section we outline the proof of Theorem \ref{th:FF};
details can be found in \cite{Forstneric2024PLMS}.
Together with the additions in subsequent sections, 
this will serve to prove the new results stated in the Introduction.

We begin by establishing the notation.
Let $X$ be a bordered Riemann surface, and assume that $G$
is a finite group acting on $X$ by holomorphic automorphisms 
and acting on $\R^n$ by orthogonal transformations. By $\CMI(X,\R^n)$
we denote the space of all conformal minimal immersions
$X\to\R^n$. (This notation is used in 
\cite{AlarconForstnericLopez2021} and several other sources.) 
It contains the subset 
\[
	\CMI_G(X,\R^n) \subset \CMI(X,\R^n)
\]
of all $G$-equivariant conformal minimal immersions. 
Given $r\in \N=\{1,2,\ldots\}$, we denote by 
\begin{equation}\label{eq:CMIG}
	\CMI_G^r(\overline X,\R^n) \subset \CMI^r(\overline X,\R^n)
\end{equation}
the spaces of maps $X\to\R^n$ of this type which extend to immersions 
$\overline X\to\R^n$ of class $\Cscr^r$. Adding a star as a subscript 
will denote the spaces of nonflat maps of the respective types:
\begin{equation}\label{eq:CMIG*}
	\CMI_{G,*}^r(\overline X,\R^n) \subset \CMI_*^r(\overline X,\R^n).
\end{equation}

We now recall the setup used in the proof of Theorem \ref{th:FF} in 
\cite{Forstneric2024PLMS}. As explained in the introduction, 
\[ 
	X_0= \{x\in X: G_x\ne \{1\}\} 
\] 
is a closed, discrete, $G$-invariant subset of $X$, and 
$G$ acts freely on its complement 
\[
	X_1:=X\setminus X_0 =
	\{x\in X: gx\ne x\ \text{for all}\ g\in G\setminus \{1\}\}. 
\]
The orbit space $X/G$ is an open Riemann surface, the 
quotient projection 
\[
	\pi:  X \to X/G
\]
is a holomorphic map with branch locus $X_0$, and 
the restriction $\pi|_{X_1}: X_1 \to X_1/G$ is a holomorphic 
covering projection of degree $|G|$. Choose a holomorphic immersion 
$\tilde h:X/G\to \C$ (such exists by \cite{GunningNarasimhan1967}). 
The holomorphic function
\begin{equation}\label{eq:h}
	h=\tilde h\circ \pi:X\to \C
\end{equation}
is $G$-invariant (that is, $h\circ g=h$ for all $g\in G$) and its branch locus  
equals $X_0$. By the chain rule, the holomorphic $1$-form $\theta=dh$ 
on $X$ satisfies the following condition for every $g\in G$:
\begin{equation}\label{eq:theta}
	\theta_{gx} \circ dg_x= \theta_x \ \ \text{for all $x\in X$},\quad 
	\{\theta=0\}=X_0. 
\end{equation} 

The implication (b)$\Rightarrow$(a) in Theorem \ref{th:FF} is seen 
as follows. Let $F:X\to \R^n$ be a $G$-equivariant immersion. 
Differentiating the equation $g\circ F=F\circ g$ at $x\in X$ 
and using $gx=x$ gives  
\begin{equation}\label{eq:gdFx}
	g\circ dF_{x} = dF_{gx} \circ dg_{x}=dF_{x} \circ dg_{x}:
	T_{x}X \to \R^n.
\end{equation}
(The first equality holds for every $x\in X$ and $g\in G$.)
Hence, the differential $dF_x:T_xX\to \R^n$ is a $G_x$-equivariant linear
embedding onto its image $dF_{x}(T_{x}X)\subset \R^n$, and 
$dF_x$ is conformal if $F$ is such. Taking the $\C$-linear parts gives
\begin{equation}\label{eq:diFequivariant}
	g\circ \di F_{x} = \di F_{gx} \circ dg_{x}=\di F_{x} \circ dg_{x}.
\end{equation}
Here, $\di$ is the $\C$-linear component of the exterior differential $d$,
so $d=\di+\dibar$. 

In the remainder of this section, we outline the proof of the 
implication (a)$\Rightarrow$(b) in Theorem \ref{th:FF}. 
We introduce the following notation:
\begin{eqnarray}
	\bA &=& \{z=(z_1,\ldots,z_n)\in \C^n: z_1^2+z_2^2+\cdots+z_n^2=0\}
	\ \ \text{the null quadric}; 
	\label{eq:bA} 
	\\
	\bA_* &=& \bA\setminus \{0\}\ \ \text{the punctured null quadric}; 
	\label{eq:bA*} 
	\\
	\overline \bA &=& 
	\text{the closure of $\bA\subset \C^n$ in 
	the projective compactification $\CP^n=\C^n\cup \CP^{n-1}$}; 
	\label{eq:overbA} \\
	Y     &=& \overline \bA\setminus \{0\} = \bA_*\cup Y_0; 
	\label{eq:Y} \\
	Y_0 &=& Y\setminus \bA_*=\bigl\{[z_1:\cdots:z_n] \in 
	\CP^{n-1}: z_1^2+z_2^2+\cdots+z_n^2=0\bigr\}.
	\label{eq:Y0}
\end{eqnarray}
The projection 
\[
	p:\bA_*\to Y_0,\quad p(z_1,\ldots,z_n)=[z_1:\cdots:z_n]
\] 
is a holomorphic $\C^*$-bundle. Extending $p$ to $Y$ as the identity on 
$Y_0$ gives a holomorphic line bundle projection $p:Y\to Y_0$ 
with the zero section $Y_0$. 
The actions of the real and the complex orthogonal 
groups $O(n,\R)\subset O(n,\C)$ on $\C^n$ 
extend to $\CP^n$, with $\bA_*$ and the hyperplane at infinity 
$\CP^n\setminus \C^n\cong \CP^{n-1}$ as invariant submanifolds. 
To every oriented two-dimensional linear subspace $\Lambda\subset\R^n$
we associate a complex null line $L=L(\Lambda) \subset \bA\subset \C^n$ 
by choosing an oriented basis $(u,v)$ of $\Lambda$ such that 
$\|u\|=\|v\|\ne 0$ and $u\,\cdotp v=0$ (such a pair $(u,v)$ is called a 
conformal frame) and setting
\begin{equation}\label{eq:L}
	L(\Lambda) = \C(u-\imath v) \subset \bA\subset \C^n.
\end{equation}
A rotation $R_\phi$ on $\Lambda$ for the angle $\phi$ in the positive
direction corresponds to the multiplication by $\E^{\imath \phi}$ on the 
complex null line $L(\Lambda)$. 

Fix a point $x_0\in X$. By the Weierstrass representation formula,
every conformal minimal immersion $F:X\to\R^n$ is of the form
\begin{equation}\label{eq:F}
	F(x)=v + \Re \int_{x_0}^x f\theta, \qquad x\in X,
\end{equation}
where $v=F(x_0)\in\R^n$ and 
\[
	f=2\di F/\theta: X\to Y = \bA_*\cup Y_0
\]
is a holomorphic map such that $f^{-1}(Y_0)=X_0$,   
the holomorphic $1$-form $f\theta=2\di F$ has no zeros or poles, and 
it satisfies the real period vanishing conditions 
\begin{equation}\label{eq:PV}
	\Re \int_C f\theta = 0\quad \text{for every $[C]\in H_1(X,\Z)$}.
\end{equation}
If in addition $F$ is $G$-equivariant then the map 
$f:X\to Y$ is also $G$-equivariant:
\begin{equation}\label{eq:fg}
	f(gx) = \frac{2\di F_{gx}}{\theta_{gx}} 
		= \frac{2\di F_{gx}\circ dg_x}{\theta_{gx} \circ dg_x}
		\stackrel{\eqref{eq:diFequivariant},\eqref{eq:theta}}{=}
		 \frac{g\, 2\di F_x}{\theta_x} = g\, f(x), \quad g\in G.
\end{equation}
The $G$-equivariance condition on the map $F$ \eqref{eq:F} 
on the orbit $Gx_0$ also gives 
\begin{equation}\label{eq:gv}
	gv = gF(x_0) = F(gx_0)=v+ \Re \int_{x_0}^{gx_0} f\theta, 
	\quad g\in G,
\end{equation}
where the integral is taken over any path from $x_0$ to $gx_0$. 
Conversely, if $f:X\to Y$ is a $G$-equivariant holomorphic map 
such that $f\theta$ is a holomorphic $\C^n$-valued $1$-form on 
$X$ without zeros (this means that the zeros of $\theta=dh$
\eqref{eq:h} at the points of $X_0$ exactly cancel the poles of $f$)
and the period vanishing conditions \eqref{eq:PV} hold, then the map 
$F$ given by \eqref{eq:F} is a well-defined conformal minimal immersion.
In particular, the real part of the integral in \eqref{eq:F} is independent of 
the choice of a path from $x_0$ to $x$. 
Assuming in addition that $f$ satisfies condition 
\eqref{eq:gv} (ignoring the two middle terms involving $F$), the map $F$ is 
also $G$-equivariant. Indeed, given a path $\gamma:[0,1]\to X$, we have 
for any $g\in G$ that 
\begin{equation}\label{eq:ggamma}
	\int_{g\gamma} f\theta =   
		\int_0^1 f(g\gamma(t)) 
		\,\theta_{g\gamma(t)}(dg_{\gamma(t)} \dot\gamma(t)) \,dt 
		\stackrel{\eqref{eq:theta}}{=} 
		\int_0^1 g f(\gamma(t)) \, \theta_{\gamma(t)} (\dot\gamma(t)) \,dt  
		= g \int_{\gamma} f\theta,
\end{equation}
and hence
\begin{eqnarray*}
	F(gx) & \stackrel{\eqref{eq:F}}{=} &
		 v+ \Re \int_{x_0}^{gx} f\theta 
		 =  \left(v + \Re \int_{x_0}^{gx_0} f\theta \right) + 
		     \Re \int_{gx_0}^{gx} f\theta \\
		 & \stackrel{\eqref{eq:gv},\eqref{eq:ggamma}}{=} & 
		 gv + \Re \left(g \int_{x_0}^{x} f\theta \right) 
		 = g\left( v + \Re \int_{x_0}^{x} f\theta\right)  \\
		 & \stackrel{\eqref{eq:F}}{=} &  gF(x). 
\end{eqnarray*} 

We now recall the two main steps in the proof of Theorem \ref{th:FF}
given in \cite{Forstneric2024PLMS}.

%
%
\medskip
\noindent {\bf Step 1:} There is a $G$-equivariant conformal minimal immersion 
$F_0:V\to\R^n$ from a neighbourhood $V$ of the closed discrete subset 
$X_0\subset X$ \eqref{eq:X0}. 

To see this, fix a point $x_0\in X_0$ and 
set $k=|G_{x_0}|>1$. As explained in the introduction, 
there is a holomorphic coordinate $z$ 
on a disc $x_0\in \Delta\subset X$, with $z(x_0)=0$, such that 
a generator $g_0$ of $G_{x_0}$ acts by  
\[
	g_0 z=\E^{\imath\phi}z,\quad \phi=2\pi/k,\quad k=|G_{x_0}|.
\]
The function $w=z^k$ is a local holomorphic coordinate on 
$X/G$ at $\pi(x_0)$, with $\pi(z)=z^k$ and 
\[
	 h(z)-h(0)=\tilde h(z^k) - \tilde h(0), \qquad 
	 \theta=dh=\tilde h'(z^k) kz^{k-1}\,dz. 
\]
By condition (a) in Theorem \ref{th:FF} there is a
$G_{x_0}$-equivariant conformal linear isomorphism $F_0:\R^2 \to \Lambda$ onto a $2$-plane $\Lambda\subset \R^n$, so $g_0$ acts on $\Lambda$
as the rotation through the angle $\phi=2\pi/k$. 
It follows that $g_0$ acts on the null line $L=L(\Lambda)\subset\C^n$ 
\eqref{eq:L} as multiplication by $\E^{\imath\phi}$. 
Note that $2\di F_0=f_0\theta$ where 
\begin{equation}\label{eq:f0}
	f_0(z) = \frac{y_0}{\tilde h'(z^k) z^{k-1}},
	\qquad y_0\in L\setminus \{0\},\ z\in \Delta.
\end{equation}
Although this function has a pole at $\{z=0\}=\{x_0\}$
in the affine trivialisation of the null line, 
it defines a holomorphic map $f_0:\Delta\to Y$ satisfying 
\[ 
	f_0(x_0)\in Y_0, \qquad f_0(\Delta\setminus\{x_0\})\subset\bA_*. 
\] 
Here, $\bA_*$ is the punctured null quadric \eqref{eq:bA*}.
Moreover, the map $f_0$ is $G_{x_0}$-equivariant by \eqref{eq:fg}.
We extend $F_0$ and $f_0$ by $G$-equivariance to the orbit 
$G\cdotp\Delta$ and perform the same construction on 
all $G$-orbits of $X_0$. This defines a $G$-equivariant map $f_0:V\to Y$ 
on a $G$-invariant neighbourhood $V\subset X$ of $X_0$, 
with $f_0^{-1}(Y_0)=X_0$, 
and a $G$-equivariant conformal minimal immersion $F_0:V\to\R^n$ 
with $2\di F_0=f_0\theta$.


\medskip
\noindent {\bf Step 2:} 
With $F_0$ and $f_0$ as above, we construct a $G$-equivariant 
holomorphic map $f:X\to Y$, which agrees with $f_0$ to sufficiently 
high order along $X_0$, such that $f(X_1)\subset\bA_*$ \eqref{eq:bA*} 
and the conditions \eqref{eq:PV} and \eqref{eq:gv} hold. 
The interpolation order at the points of $X_0$ is chosen so that the 
holomorphic $1$-form $f\theta$ extends across $X_0$ without zeros. 
The map $F:X\to\R^n$ defined by \eqref{eq:F} is then a $G$-equivariant
conformal minimal immersion. In addition, we ensure that 
$F|_{X_0}=F_0|_{X_0}$.

Let $Y$ be given by \eqref{eq:Y}.
Consider the action of $G$ on $X\times Y$ given by 
\[
	g(x,y)=(gx,gy),\quad x\in X,\ y\in Y,\ g\in G.  
\]
The natural projection $X\times Y\to X$ is then $G$-equivariant, 
so it induces a projection 
\begin{equation}\label{eq:Z}
	\rho:Z=(X\times Y)/G \to X/G. 
\end{equation}
Clearly, $G$-equivariant maps $f:X\to Y$ are naturally identified
with sections $\tilde f:X/G \to Z$ of $\rho:Z \to X/G$. 
Note that $Z=(X\times Y)/G$ is a reduced complex space, 
the map $\rho$ \eqref{eq:Z} is holomorphic, 
it is branched over the closed discrete subset $X_0/G$ of $X/G$, 
and the restricted map 
\[
	\rho: Z_1=\rho^{-1}(X_1/G) \to X_1/G,
\] 
where $X_1=X\setminus X_0$, 
is a holomorphic fibre bundle with the fibre $Y=\bA_*\cup Y_0$;
see \eqref{eq:bA*}--\eqref{eq:Y0}. 
(The transition maps of this bundle are given by the 
action of $G$ on $Y$.) Since the submanifold $Y_0 \subset Y$ 
and its complement $Y\setminus Y_0=\bA_*$ are both $G$-invariant, 
$Z_1$ is the disjoint union 
\[
	Z_1 = \left[ (X_1\times \bA_*)/G\right] 
	\sqcup \left[(X_1\times Y_0)/G \right].
\]
The subset  
\[ 
	\Omega:=(X_1\times \bA_*)/G 
\] 
of $Z_1$ is a $G$-invariant Zariski open domain in $Z$ without singularities, 
and its complement $Z\setminus \Omega$ is a closed complex 
subvariety of $Z$ containing $\mathrm{sing}(Z) \cup \mathrm{br}(\rho)$.
The restricted projection 
\begin{equation}\label{eq:rho} 
	\rho:\Omega=(X_1\times \bA_*)/G \to X_1/G
\end{equation} 
is a holomorphic $G$-bundle with fibre $\bA_*$.
The main point of the above construction is summarised
in the following result. 

%
%
\begin{lemma}[Lemma 3.1 in \cite{Forstneric2024PLMS}]
\label{lem:correspondence}
There is a natural bijective correspondence between (continuous or holomorphic)
sections $\tilde f:X/G \to Z$ of the map $\rho:Z\to X/G$ \eqref{eq:Z}
satisfying 
\[
	\tilde f(X_1/G)\subset \Omega=(X_1\times \bA_*)/G
	\ \ \text{and}\ \ 
	\tilde f(X_0/G) \subset (X_0\times Y_0)/G
\]
and (continuous or holomorphic) $G$-equivariant maps $f:X\to Y$
satisfying 
\[
	f(X_1)\subset \bA_*\ \ \text{and}\ \ f(X_0)\subset Y_0.
\]
\end{lemma}

The map $f_0$ from step 1 of the proof corresponds to 
a local holomorphic section $\tilde f_0$ of $Z\to X/G$ on 
an open neighbourhood $V/G \subset X/G$ of $X_0/G$ such that  
\[
	\tilde f_0\bigl((V \setminus X_0)/G)\bigr)  \subset \Omega.
\]
The fibre $\bA_*$ of the holomorphic fibre bundle \eqref{eq:rho} is 
$O(n,\C)$-homogeneous, hence an Oka manifold by a theorem
of Grauert \cite{Grauert1958MA}. (See also 
\cite[Proposition 5.6.1]{Forstneric2017E}.) 
Therefore, sections of the ramified holomorphic map 
$\rho:Z=(X\times Y)/G \to X/G$ mapping 
$X_1/G$ to $\Omega$ satisfy the Oka principle.  
(See \cite[Theorem 6.14.6, p.\ 310]{Forstneric2017E}, 
which originates from \cite{Forstneric2003FM}, along with 
\cite[Theorem 1.3]{Forstneric2023Indag}.) 
This gives global holomorphic sections $\tilde f:X/G\to Z$ satisfying 
$\tilde f(X_1/G)\subset \Omega$ which agree with $\tilde f_0$ 
at every point $p\in X_0/G$ to any given finite order $k(p)\in \N$.
Furthermore, $\tilde f$ can be chosen such that the corresponding
$G$-equivariant map $f:X\to Y$ integrates to a $G$-equivariant conformal
minimal immersion; the necessary period vanishing conditions are obtained 
using period dominating sprays.
The key technical ingredient in the construction 
of a section $\tilde f$ with the desired properties 
is provided by \cite[Lemma 4.5]{Forstneric2024PLMS},
and the construction is completed in \cite[Sect.\ 5]{Forstneric2024PLMS}.

A more precise version of Theorem \ref{th:FF}, 
with approximation on admissible Runge sets and interpolation on closed
discrete subsets, is \cite[Theorem 5.1]{Forstneric2024PLMS}.
We state it here for convenience of the reader since we will refer to it 
in the sequel. We begin by recalling the relevant definitions.

%
%
\begin{definition}
[Definition 1.12.9 in \cite{AlarconForstnericLopez2021}] 
\label{def:admissible}
An {\em admissible set} in a Riemann surface $X$ is 
a compact set of the form $S=K\cup E$, where $K$ is a 
(possibly empty) finite union of pairwise disjoint compact domains with piecewise 
$\Cscr^1$ boundaries in $X$ and $E = \overline{S \setminus K}$ is a union of 
finitely many pairwise disjoint, smoothly embedded Jordan arcs and 
closed Jordan curves meeting $K$ only at their endpoints (if at all) and 
such that their intersections with the boundary $bK$ of $K$ are transverse. 
\end{definition}

%
%
\begin{definition}[Definition 3.1.2 in \cite{AlarconForstnericLopez2021}] 
\label{def:GCMI}
Let $S=K\cup E$ be an admissible set in a Riemann surface $X$,
let $\theta$ be a holomorphic $1$-form 
on a neighbourhood of $S$ in $X$ without zeros on $bS$,
and let $Y=\overline \bA\setminus \{0\}$ \eqref{eq:Y}.  
A {\em generalised conformal minimal immersion} $S\to\R^n$ $(n\ge 3)$ 
of class $\Cscr^r$ $(r\in\N)$ is a pair $(F,f\theta)$, where $F: S\to \R^n$ 
is a $\Cscr^r$ map whose restriction to 
$\mathring S=S\setminus bS=\mathring K$ 
is a conformal minimal immersion and the map $f\in \Ascr^{r-1}(S,Y)$ 
satisfies the following conditions:
\begin{enumerate}[\rm (a)]
\item 
$f\theta =2\di F$ holds on $K$ (in particular, the zeros of $\theta$ 
cancel the poles of $f$), and
\item 
for any smooth path $\alpha:[0,1]\to X$ parameterizing a connected 
component of $E=\overline{S\setminus K}$ we have that 
$\Re(\alpha^*(f\theta))=\alpha^* dF = d(F\circ \alpha)$.
\end{enumerate}
\end{definition}

Given an admissible set $S\subset X$ and integers $r\ge 1,\ n\ge 3$, 
we denote by 
\[
	\GCMI^r(S,\R^n)
\]
the space of all generalised conformal minimal immersions $S\to\R^n$ 
of class $\Cscr^r$.  
An element $(F,f\theta)\in\GCMI^r(S,\R^n)$ is called \emph{nonflat} if, on every connected component of $K$ and of $E$, the image of $f$ is not contained in the completion in $Y$ of a complex line through the origin in $\bA$.

Let $G$ be a finite group as in Theorem \ref{th:FF}, 
$S\subset X$ be a $G$-invariant admissible set, and $\theta$ 
be a $G$-invariant 1-form on $X$ (see \eqref{eq:theta}). 
The pair $(F,f\theta)\in\GCMI^r(S,\R^n)$ is called $G$-equivariant if 
\[ 
	F(gx)=gF(x),\quad f(gx)=gf(x), 
	\quad x\in S,\ g\in G. 
\]
If $S=K$ (i.e., there are no attached arcs) then $G$-equivariance
of $F$ implies $G$-equivariance of $f$ by \eqref{eq:fg}; however,
this is not necessarily the case on the attached arcs. 
Conversely, $G$-equivariance of $f$ implies $G$-equivariance of its 
integral after the integration constants have been chosen compatibly. 
On a connected set, this is ensured by the orbit-integral conditions 
\eqref{eq:gv}. We denote by 
\begin{equation}\label{eq:GCMIG}
	\GCMI^r_{G,*}(S,\R^n) \subset \GCMI^r_G(S,\R^n)
\end{equation}
the spaces of (nonflat) $G$-equivariant generalised 
conformal minimal immersions $S\to\R^n$ of class $\Cscr^r$. 
We recall the following more precise version of Theorem \ref{th:FF} 
which will be used. 

%
%
\begin{theorem}[Theorem 5.1 in \cite{Forstneric2024PLMS}]
\label{th:FFbis}
Assume that $G$ is a finite group acting effectively on a connected open 
Riemann surface $X$ by holomorphic automorphisms and acting on 
$\R^n$, $n\ge 3$, by orthogonal transformations. Let $\pi:X\to X/G$ be 
the quotient projection and $X_0\subset X$ be the set \eqref{eq:X0} 
of its branch points. 
Assume that $S\subset X$ is a $G$-invariant admissible set 
(see Definition \ref{def:admissible}) which is Runge in $X$,  
$A\subset X$ is a closed $G$-invariant discrete set such that 
$X_0\subset A\subset X\setminus bS$, 
$V\subset X\setminus S$ is an open $G$-invariant neighbourhood of 
$A\setminus S$, and $F_0:S\cup V\to \R^n$ is such that 
$F_0|_S\in \GCMI^r_{G,*}(S,\R^n)$ $(r\ge 1)$ 
and $F_0|_V\in \CMI_G(V,\R^n)$.
Then, $F_0$ can be approximated in the $\Cscr^r$ 
topology on $S$ by nonflat $G$-equivariant conformal minimal immersions 
$F:X\to\R^n$ which agree with $F_0$ to a given finite order $k(a)\in\N$ 
at every point $a\in A$.
\end{theorem}

The construction described above, which was developed in 
\cite{Forstneric2024PLMS}, also yields the following theorem concerning 
the existence of $G$-equivariant conformal minimal surfaces with branch 
points. This result uses the weaker hypothesis (a') in Remark
\ref{rem:cond-a} on the points $x\in X$ with nontrivial stabiliser $G_x$.

\begin{theorem} \label{th:branched}
Let $G$ be a finite group acting effectively on a connected open 
Riemann surface $X$ by holomorphic automorphisms and acting on 
the Euclidean space $\R^n$, $n\ge 3$, by orthogonal transformations. 
Assume that for every nontrivial stabiliser $G_x$, $x\in X$, there is a 
two-dimensional linear subspace $\Lambda_x\subset\R^n$ on which 
$G_x$ acts effectively as a cyclic group of rotations. 
Then, there exists a $G$-equivariant 
conformal harmonic map $F:X\to\R^n$ which is branched at every 
point $x\in X$ at which the stabiliser group $G_x$ is nontrivial
but does not satisfy condition (a) in Theorem \ref{th:FF}.
\end{theorem}

\begin{proof}
The only difference concerns the local construction at the points 
$x_0\in X_0$ \eqref{eq:X0} with a nontrivial stabiliser group
$G_{x_0}$; see Step 1 above. Choose a generator 
$g_0$ of $G_{x_0}$ such that, in a suitable local holomorphic
coordinate $z$ on $X$ with $z(x_0)=0$, we have
$g_0z=\E^{\imath\phi}z$ where $\phi=2\pi/k$, $k=|G_{x_0}|$
(see \eqref{eq:rotation}). 
If condition (a) in Theorem \ref{th:FF} holds at $x_0$, choose the unbranched 
local model from Step 1 above. Otherwise, choose a two-dimensional linear 
subspace $\Lambda\subset\R^n$ on which $G_{x_0}$ acts effectively 
by rotations. There is an orientation of $\Lambda$ in which the generator 
$g_0$ acts on $\Lambda$ as the rotation through the angle 
$m\phi$ for some integer $m$ satisfying 
\[
	1\le m \le k/2, \qquad \gcd(m,k)=1. 
\] 
Choose an oriented conformal frame $(\bu,\bv)$ for $\Lambda$. The map 
\[
	F_0(z)=F_0(r \E^{\imath t}) = 
	r^m \left(\cos(mt)\bu + \sin(mt)\bv\right)
	=\Re\bigl(z^m(\bu-\imath\bv)\bigr)
\]
is then a $G_{x_0}$-equivariant conformal harmonic map 
whose branch order at $z=0$ is $m-1$. In particular, $F_0$ is unbranched 
(a linear isomorphism onto $\Lambda$) if and only if $m=1$.
Extending this construction by $G$-equivariance to the orbit
$Gx_0$ and doing the same on every $G$-orbit in $X_0$
gives a conformal $G$-equivariant harmonic map $F_0:V\to\R^n$
from a $G$-invariant neighbourhood $V\subset X$ of $X_0$ 
whose branch locus is precisely the subset of $X_0$ where condition 
(a) in Theorem \ref{th:FF} fails. In Step 2 of the proof, impose sufficiently 
high jet interpolation on the local maps $f_0=2\di F_0/\theta$ 
at the points of $X_0$. This preserves their pole orders and nonzero leading coefficients. Consequently, the resulting differential $f\theta$ has the prescribed zero orders at these points and is nowhere zero on $X\setminus X_0$. 
The remaining details in Step 2 are the same as above.
\end{proof}

%
%
%
%
\section{Main new tools}
\label{sec:tools}
In this section, we present the main new tools that, together with the results 
from  \cite{Forstneric2024PLMS} in the previous section, will be used in the 
proof of the main theorems. Namely, we establish $G$-equivariant versions 
of the method for exposing boundary points of bordered Riemann surfaces, 
due to Forstneri\v c and Wold \cite{ForstnericWold2009} (see Theorem 
\ref{th:exposing} and Corollary \ref{cor:curves}), as well as of the 
Riemann-Hilbert method for minimal surfaces developed by 
the authors \cite{AlarconForstneric2015MA}, and in a more complete
form by Alarc\'on, Drinovec Drnov\v sek, Forstneri\v c, and L\'opez 
\cite{AlarconDrinovecForstnericLopez2015PLMS} 
(see Theorem \ref {th:RH}). 

%
%
%
\subsection{$G$-equivariant exposure of boundary points of bordered 
Riemann surfaces}
\label{sec:exposing}
The following result is an equivariant version 
of \cite[Theorem 6.7.1]{AlarconForstnericLopez2021}. 
The technique used in the proof originates in the 
result of Forstneri{\v c} and Wold \cite[Theorem 2.3]{ForstnericWold2009}
on exposing boundary points of bordered Riemann surfaces;
see also \cite[Theorem 9.9.1]{Forstneric2017E}. 
This method was first used to prove that if 
$\overline X$ is a compact bordered Riemann surface which admits 
a holomorphic embedding in $\C^2$, then the interior $X$ admits 
a proper holomorphic embedding in $\C^2$ \cite{ForstnericWold2009}. 
Later, it played a fundamental role 
in showing that every bordered Riemann surface 
immerses properly holomorphically in the ball of $\C^2$ as a complete 
complex curve \cite{AlarconForstneric2013MA}, as well as in all 
constructions of complete bounded minimal surfaces normalised by given bordered Riemann surfaces (see the discussion in Section \ref{sec:intro}). 

%
%
\begin{theorem}  \label{th:exposing}
Let $G$ be a finite group acting effectively on a connected open 
Riemann surface $R$ by holomorphic automorphisms.
Assume that $X\Subset R$ is a smoothly bounded 
$G$-invariant domain such that $G$ acts freely on $bX$,
$A=\{a_1,\ldots, a_k\} \subset X$ and 
$P=\{p_1,\ldots, p_m\}\subset bX$ are $G$-invariant subsets, 
and $E_1,\ldots,E_m$ are smoothly embedded pairwise disjoint compact 
arcs in $R$ such that for each $i=1,\ldots, m$, 
$E_i$ is attached to $\overline X$ 
transversely at the point $p_i\in bX$, $E_i\cap \overline X=\{p_i\}$,
and the set $E=\bigcup_{i=1}^m E_i$ is $G$-invariant.
(See \cite[Fig.\ 6.4, p.\ 292]{AlarconForstnericLopez2021}.)
Let $q_i\in E_i$ denote the other endpoint of $E_i$, set 
$Q=\{q_1,\ldots,q_m\}$, and let $b_i$ 
be a point in the relative interior of $E_i$ such that the set
$B=\{b_1,\ldots,b_m\}$ is $G$-invariant.
Suppose that for every $i=1,\dots,m$ we are given 
a neighbourhood $V_i\subset R$ of the arc $E_i$ and a 
neighbourhood $U_i\Subset V_i$ of the point $p_i$ such that the sets
$\overline V_i$ $(i=1,\ldots,m)$ 
are pairwise disjoint and the sets $U=\bigcup_{i=1}^m U_i$
and $V=\bigcup_{i=1}^m V_i$ are $G$-invariant. 
Given $d \in \N$ there exists a smooth $G$-equivariant
diffeomorphism $\phi:  \overline X \to  \phi(\overline X)\subset R$ 
which is holomorphic in $X$ and satisfies the following conditions.
\begin{enumerate}[\rm (a)]
\item $\phi$ approximates the identity to a desired precision in the 
smooth topology on $\overline X \setminus U$.
\item $\phi$ is tangent to the identity map to order $d$ at each point $a_1,\ldots,a_k$.
\item For every $i=1,\ldots,m$ we have that
\begin{equation}\label{eq:exposing}
	\phi(\overline X\cap U_i)\subset V_i,
	\quad \ \phi(p_i) = q_i,  
	\quad\ b_i\in \phi(X).
\end{equation}
\end{enumerate}
\end{theorem}
\begin{proof} 
Let $\pi:R\to R/G$ be the quotient projection, and set 
\[ 
	X_0=\{x\in X:G_x\ne\{1\}\}. 
\] 
First observe that $G$ acts freely on the union $E=\bigcup_{i=1}^m E_i$.  
Indeed, if $gx=x$ for some $x\in E_i$, then the disjointness of the arcs implies 
$gE_i=E_i$. Since $E_i\cap\overline X=\{p_i\}$ and $\overline X$ is 
$G$-invariant, this gives $gp_i=p_i$. The freeness of the action on $bX$ 
then implies $g=1$. Consequently, after replacing $R$ by a sufficiently 
small $G$-invariant neighbourhood of $\overline X\cup E$, we may 
assume that the ramification locus of $\pi$ is precisely the finite set $X_0$. Enlarging the finite set $A$ if necessary, we may assume that $X_0\subset A$. 
The domain $X/G\Subset R/G$ has smooth boundary, and 
the projection 
\[ 
	\pi:R\setminus X_0\longrightarrow (R\setminus X_0)/G 
\] 
is an unbranched covering with deck group $G$. Choose one arc $E_i$ 
from each $G$-orbit of the family $\{E_1,\ldots,E_m\}$. Their images 
under $\pi$ are pairwise disjoint embedded arcs attached transversely 
to $b(X/G)$. By shrinking the neighbourhoods used in the construction,
we may choose, for these representative arcs, pairwise disjoint simply 
connected neighbourhoods $T_i\subset (R\setminus X_0)/G$ 
such that $\pi(E_i)\subset T_i$ and the component of $\pi^{-1}(T_i)$ 
containing $E_i$ is contained in $V_i$. We also choose sufficiently small 
boundary neighbourhoods of $\pi(p_i)$ whose lifts containing $p_i$ lie in 
$U_i$. These choices can be made disjoint from $\pi(A)$. 
Applying \cite[Theorem 6.7.1]{AlarconForstnericLopez2021} to the domain
$X/G\Subset R/G$, the projected arcs, 
and the interpolation set $\pi(A)=A/G$ gives a smooth embedding 
\[ 
	\varphi:\overline{X}/G \longrightarrow R/G 
\] 
which is holomorphic in $X/G$, is as close to the identity as desired away from the chosen boundary neighbourhoods, and satisfies
\[
	\varphi(\pi(p_i))=\pi(q_i),\qquad \pi(b_i)\in \varphi(X/G). 
\] 
The images of the chosen boundary neighbourhoods lie in the corresponding 
sets $T_i$. We require, in addition, that $\varphi$ agree with the identity to 
order $N$ at every point of $\pi(A)=A/G$, where $N\ge d+1$. We explain 
how to choose the lift of $\varphi$. There is a homotopy 
\[ 
	H_t:\overline{X}/G \longrightarrow R/G, \qquad 0\le t\le1, 
\] 
from $H_0=\Id$ to $H_1=\varphi$ such that $H_t|_{X_0/G}=\Id$
and $H_t^{-1}(X_0/G)=X_0/G$ for all $t\in [0,1]$, the homotopy is uniformly 
close to the identity away from the chosen boundary neighbourhoods, and 
\[ 
	H_t(\pi(p_i))\in T_i,\qquad 0\le t\le1. 
\] 
Furthermore, near every point $c\in \pi(A)$, $H_t$ may be chosen holomorphic 
and tangent to the identity to order $N$. Here is a direct construction of this homotopy. On the complement of small boundary half-discs around the 
points $\pi(p_i)$, the map $\varphi$ is arbitrarily close to the identity, so 
the two maps can be joined by a small homotopy. In a coordinate $w$ 
centred at $c\in\pi(A)$, choose this homotopy to be 
\[ 
	H_t(w)=w+t\bigl(\varphi(w)-w\bigr). 
\] 
On a small coordinate disc this is holomorphic, fixes $c$, 
and has no other preimage of $c$. On the remaining part 
the homotopy can be chosen to agree with these local homotopies. 
Over each boundary half-disc, extend the homotopy inside the corresponding 
contractible set $T_i$. As the sets $T_i$ avoid $X_0/G$, this construction 
has all the stated properties. The covering homotopy property 
now gives a unique lift 
\[ 
	\phi_t:\overline X\setminus X_0\longrightarrow R\setminus X_0,
	\quad t\in [0,1], 
\] 
such that 
\[ 
	\phi_0=\Id,\qquad \pi\circ\phi_t=H_t\circ\pi,\quad t\in [0,1]. 
\] 
For every $g\in G$, the maps $\phi_t\circ g$ and $g\circ\phi_t$ lift the 
same homotopy and agree at $t=0$. Uniqueness of homotopy lifting 
therefore gives 
\[ 
	\phi_t\circ g=g\circ\phi_t, \qquad g\in G,\ t\in [0,1]. 
\] 
In particular, the terminal lift $\phi=\phi_1$ is $G$-equivariant. 
We next check its extension across $X_0$ and the interpolation condition. 
Let $a\in X_0$ and $k=|G_a|$. Choose local coordinates $z$ at $a$ and 
$w$ at $\pi(a)$ in which $\pi$ is given by $w=z^k$. The prescribed 
interpolation gives 
\[ 
	\varphi(w)=w+w^{N+1}\alpha(w) 
\] 
for a holomorphic function $\alpha$ near $0$. The local homotopy above is 
\[ 
	H_t(w)=w\bigl(1+t w^N\alpha(w)\bigr), \quad t\in [0,1]. 
\] 
Its lift starting at the identity is 
\[ 
	\phi_t(z) =z\bigl(1+t z^{kN}\alpha(z^k)\bigr)^{1/k}, \quad t\in [0,1],
\] 
where the root is chosen continuously from the value $1$ at $t=0$. 
On a sufficiently small disc this root is holomorphic and equals $1$ at $z=0$. Hence $\phi$ extends holomorphically across $a$ and satisfies 
\[ 
	\phi(z)=z+O(z^{kN+1}). 
\] 
In particular, it is locally biholomorphic there and agrees with the identity to 
order at least $d$. At a point of $A\setminus X_0$, the same conclusion 
follows from the fact that $\pi$ is locally biholomorphic and the lift is the one 
close to the identity. Thus condition {\rm(b)} holds. Since $\varphi$ is 
holomorphic in $X/G$ and smooth up to its boundary, its lift 
$\phi$ is holomorphic in $X$ and smooth on $\overline X$. 
We claim that it is injective. Indeed, if $\phi(x)=\phi(y)$ 
then $\varphi(\pi(x))=\varphi(\pi(y))$, and hence 
$\pi(x)=\pi(y)$. Thus $y=gx$ for some $g\in G$. Equivariance of $\phi$ gives 
$g\phi(x)=\phi(gx)=\phi(y)=\phi(x)$. 
Away from $X_0$ the action is free, so $g=1$. At points of $X_0$, 
injectivity follows because the lift fixes every such point. 
Therefore $\phi$ is a smooth diffeomorphism onto its image. 
The smallness of the homotopy away from the boundary neighbourhoods 
selects the lift close to the identity there. Smooth dependence of local 
inverse branches of $\pi$, together with the preceding local formula at 
ramification points, shows that $\phi$ approximates the identity as closely 
as desired in the smooth topology on $\overline X\setminus U$. 
This proves {\rm(a)}. Finally, the path $t\mapsto H_t(\pi(p_i))$ lies in $T_i$. 
Its lift starting at $p_i$ remains in the component of $\pi^{-1}(T_i)$ 
containing $E_i$. Since $\pi$ is injective on this component and 
$\varphi(\pi(p_i))=\pi(q_i)$, we obtain 
$
	\phi(p_i)=q_i. 
$ 
The same choice of sheets and the neighbourhood control in the exposing construction give 
$
	\phi(\overline X\cap U_i)\subset V_i 
$,  
using smaller boundary neighbourhoods inside $U_i$ and sufficiently close 
approximation on their complement. The image $\phi(X)$ is 
$G$-invariant and satisfies $\pi(\phi(X))=\varphi(X/G)$. Consequently, 
\[ 
	\phi(X)=\pi^{-1}\bigl(\varphi( X/G)\bigr). 
\] 
Since $\pi(b_i)\in\varphi( X/G)$, this implies $b_i\in\phi(X)$. Thus all three 
conditions in \eqref{eq:exposing} hold. 
\end{proof}

Recall that $E=\bigcup_{i=1}^m E_i$. 
In the situation of Theorem \ref{th:exposing}, the set 
\begin{equation}\label{eq:S}
	S= \overline X \cup E \subset R
\end{equation}
is a $G$-invariant admissible set (see Definition \ref{def:admissible}).
Recall that $q_i$ denotes the endpoint of the arc $E_i$ which is not contained 
in $\overline X$. Combining Theorems \ref{th:FFbis} 
and \ref{th:exposing} gives the following corollary, 
which is used in the proof of all main results in the paper.

%
%
\begin{corollary}\label{cor:curves}
Assume the hypothesis of Theorem \ref{th:exposing}, let $S\subset R$
be the admissible set \eqref{eq:S}, and suppose that $G$ acts by orthogonal
transformations on $\R^n$, $n\ge 3$. Given 
\begin{itemize}
\item $F_0\in \GCMI^r_{G,*}(S,\R^n)$, $r\ge 1$ (see \eqref{eq:GCMIG}), 
\item a neighbourhood $W_i\subset\R^n$ of the 
arc $F_0(E_i)$ and a neighbourhood $U'_i\subset R$ of the point $p_i$
for every $i=1,\ldots,m$, and 
\item numbers $\epsilon>0$ and $d\in\N$, 
\end{itemize}
there exist pairwise disjoint neighbourhoods $U_i\subset U'_i$ of 
$p_i$ $(i=1,\ldots,m)$ and a nonflat $G$-equivariant conformal minimal 
immersion $F:\overline X\to \R^n$ such that $U=\bigcup_{i=1}^m U_i$ is 
$G$-invariant and the following conditions hold:
\begin{enumerate}[\rm (a)] 
\item $\|F-F_0\|_{\Cscr^r(\overline X\setminus U)} < \epsilon$.
\item 
For every $i=1,\ldots,m$ we have 
$F(\overline X\cap U_i)\subset W_i$, $F(p_i) = F_0(q_i)$   
and $F_0(b_i) \in F(X)$.
\item $F$ agrees with $F_0$ to order $d$ on the finite set $A\subset X$.
\item $\Flux_F=\Flux_{F_0}$.
\end{enumerate}
\end{corollary}

Intuitively speaking, we make a small modification of 
$F_0:\overline X\to \R^n$ 
away from the points $p_i\in bX$ and a big modification 
in a small neighbourhood $U_i$ of every $p_i$ to obtain a 
$G$-equivariant conformal minimal immersion $F:\overline X\to\R^n$ 
which maps $U_i\cap \overline X$ into a thin tube around the arc $F_0(E_i)$ 
and sends the point $p_i\in bX$ to the other endpoint $F_0(q_i)$ 
of $F_0(E_i)$ for every $i=1,\ldots,m$.

%
%
\begin{proof} 
Choose a sufficiently small $G$-invariant open neighbourhood $\widetilde S$ of 
$S$ such that $S/G$ is a strong deformation retract of $\widetilde S/G$. 
By \cite[Lemma 4.5]{Forstneric2024PLMS}, we can approximate $F_0$ in the 
$\Cscr^r(S)$ topology, with value interpolation on $B\cup P\cup Q$ and 
interpolation to order $d$ on $A$, by a $G$-equivariant conformal minimal 
immersion $\widetilde F_0:\widetilde S\to \R^n$,   
defined on an open neighbourhood $\widetilde S\subset R$ of $S$,
with $\Flux_{\widetilde F_0}=\Flux_{F_0}$. 
Assuming that the approximation is sufficiently close on $S$, 
$\widetilde F_0$ is nonflat and satisfies 
\[ 
	\widetilde F_0(E_i)\subset W_i,\qquad i=1,\ldots,m. 
\] 
Choose pairwise disjoint neighbourhoods $V_i\subset\widetilde S$ 
of the arcs $E_i$ such that $V=\bigcup_{i=1}^m V_i$ is $G$-invariant and
\[ 
	\widetilde F_0(V_i)\subset W_i, \qquad i=1,\ldots,m. 
\] 
For every $i$, choose a neighbourhood $U_i\Subset U'_i\cap V_i$ 
of $p_i$ such that $U=\bigcup_{i=1}^m U_i$ is $G$-invariant. Let 
\[ 
	\phi:\overline X\to \phi(\overline X)\subset\widetilde S 
\] 
be a $G$-equivariant diffeomorphism furnished by Theorem \ref{th:exposing},
approximating the identity map on $\overline X\setminus U$. 
Assuming that the approximations in both steps are sufficiently close, 
the map 
\[
	F=\wt F_0 \circ \phi: \overline X\to \R^n
\]
is a well-defined $G$-equivariant conformal minimal immersion 
satisfying the corollary. The last condition in (b) 
is seen as follows. For every $i$, condition \eqref{eq:exposing} gives a point 
$x_i\in X$ such that $\phi(x_i)=b_i$. Since $\widetilde F_0$ agrees with 
$F_0$ at $b_i$, we have $F(x_i) = \widetilde F_0(b_i) = F_0(b_i)$ and  
hence $F_0(b_i)\in F(X)$.
\end{proof}

%
%
%
%
\subsection{$G$-equivariant Riemann--Hilbert deformations}
\label{sec:RH}

The classical Riemann--Hilbert boundary value problem asks for 
continuous or smooth functions on a 
compact bordered Riemann surface $\overline X$ 
which are holomorphic in the interior $X$ and whose
value at any boundary point $x\in bX$ lies in a given real curve 
$\lambda_x\subset \C$ depending continuously or smoothly on $x$.  
A general result of this type is due to \v Cerne \cite{Cerne2004}
whose paper contains references to many earlier works.
However, in geometric applications it often suffices to use 
approximate solutions of a Riemann--Hilbert problem, which 
are typically easier to obtain and which may exist even when 
exact solutions do not. Such approximate solutions
have been used in particular for moving the boundary of the 
image toward the boundary of a target domain, 
which leads by an inductive process to existence of proper maps 
(see \cite{DrinovecForstneric2007DMJ} and the references therein), 
and to increase intrinsic distances of the image 
by inserting oscillatory behaviour near 
the boundary, which inductively leads to completeness of the limit map
(see \cite{AlarconForstneric2013MA}).
The Riemann--Hilbert method was first adapted 
for use in the theory of conformal minimal surfaces in $\R^n$ and 
holomorphic null curves by the authors in \cite{AlarconForstneric2015MA}.
A comprehensive discussion of this topic can be found in 
\cite[Chap.\ 6]{AlarconForstnericLopez2021}.

We shall adapt the Riemann--Hilbert deformation method
to $G$-equivariant conformal minimal immersions.
As before, $\bA$ denotes the null quadric \eqref{eq:bA} in $\C^n$
for $n\ge 3$. Set 
\[
	\D=\{z\in \C:|z|<1\},\quad \T=b\D=\{z\in \C:|z|=1\}.
\]
The following result is a $G$-equivariant 
analogue of \cite[Theorem 6.6.2]{AlarconForstnericLopez2021};
see also \cite[Theorem 6.6.1]{AlarconForstnericLopez2021}
for the corresponding result concerning holomorphic null curves.

%
%
%
\begin{theorem}[Equivariant Riemann--Hilbert deformations
of minimal surfaces]
\label{th:RH}
Let $X$ be a bordered Riemann surface on which a finite group $G$ 
acts by holomorphic automorphisms 
such that the action is free on $bX$. Assume that $G$ also acts on 
$\R^n$ by orthogonal transformations. Let $I_1,\ldots,I_k$ be pairwise disjoint 
closed arcs in $bX$, none of which is a connected component of $bX$, 
such that $I=\bigcup_{j=1}^k {I_j}$ is $G$-invariant. 
Choose a $G$-invariant annular neighbourhood $A\subset \overline X$ 
of $bX$ and a smooth $G$-equivariant retraction $\rho:A\to bX$. 
Assume the following:
\begin{itemize}
\item $F: \overline X\to\R^n$ is a nonflat $G$-equivariant 
conformal minimal immersion of class $\Cscr^1(\overline X)$.
\item 
For each $i=1,\ldots, k$, $(\bu_i,\bv_i)$ is a pair of
orthogonal vectors in $\R^n$ satisfying $|\bu_i|=|\bv_i|>0$, and 
if for some $g\in G$ we have $g(I_i)=I_j$ then 
$g\bu_i=\bu_j$ and $g\bv_i=\bv_j$.  
\item $r :bX \to [0,1]$ is a continuous $G$-invariant function 
supported on $I$.
\item  $\sigma :I \times\overline{\D}\to\C$ is a 
$\Cscr^1$ function which is $G$-invariant in the first variable
(that is, $\sigma(g x,\xi)=\sigma(x,\xi)$ for all $g\in G$, $x\in I$, and 
$\xi\in \overline{\D}$), for every $x\in I$ the function 
$\sigma(x,\cdotp)$ is holomorphic in $\D$ and satisfies $\sigma(x,0)=0$, 
and $\frac{\di \sigma}{\di \xi}$ is nowhere vanishing on $I \times \overline\D$.
\item $\Lambda\subset X$ is a finite $G$-invariant set containing $X_0$ 
\eqref{eq:X0}.
\end{itemize}
Consider the continuous map
$\varkappa:bX \times\overline{\D}\to\R^n$ given by
\begin{equation}\label{eq:kappa}
	\varkappa(x,\xi)=
	\begin{cases}
		F(x),  & \text{if}\ x \in bX\setminus I; \\
		F(x) + \Re \sigma(x,r(x)\xi) \bu_i+ \Im \sigma(x,r(x)\xi)  \bv_i, &
		\text{if}\ x \in I_i,\; i\in\{1,\ldots,k\}.
	\end{cases}
\end{equation}
Given $\epsilon>0$, $d\in\N$, and a neighbourhood $U\subset \overline X$ 
of $I$, there exist a neighbourhood $\Omega\subset U\cap A$ of $I$ 
and a $G$-equivariant conformal minimal immersion 
$\wt F:\overline X\to \R^n$ satisfying the following conditions.
\begin{enumerate}[\rm (i)]
\item $\dist(\wt F(x),\varkappa(x,\T))<\epsilon$ for all $x\in bX$.
\item $\dist(\wt F(x),\varkappa(\rho(x),\overline{\D}))<\epsilon$ for all 
$x \in \Omega$.
\item $\wt F$ is $\epsilon$-close to $F$ in the $\Cscr^1$ norm on 
$\overline X\setminus \Omega$.
\item $\wt F-F$ vanishes to order $d$ at every point in $\Lambda$.
\item $\Flux_{\wt F}=\Flux_F$ (see \eqref{eq:flux}).
\end{enumerate}
\end{theorem}

Observe that since the action of $G$ on $bX$ is free and every element of $G$ 
preserves the boundary orientation, the induced action of $G$ on the set 
of arcs $\{I_1,\ldots,I_k\}$ in the above theorem is free. 
Indeed, if $g\in G\setminus\{1\}$ preserved a proper boundary arc setwise, 
then it would preserve the ordered pair of its endpoints. 
Hence both endpoints would be fixed by $g$, contrary to the 
freeness of the action on $bX$.
%
%
\begin{proof} 
Let $(\bu_i,\bv_i)$ for $i=1,\ldots,k$ be as in the theorem. 
Then, 
\[
	\bw_i=\bu_i-\imath\bv_i\in \bA_*
	\subset \C^n\setminus\{0\}, \qquad i=1,\ldots,k
\]
Furthermore, if $g\in G$ and $g(I_i)=I_j$ then 
\[ 
	g\bw_i=g\bu_i-\imath g\bv_i =\bu_j-\imath\bv_j=\bw_j. 
\] 
Thus the family of null directions $\{\bw_i\}_{i=1}^k$ 
is compatible with the action of $G$. 

We will adapt the proof of the Riemann--Hilbert theorem 
for conformal minimal immersions 
(see \cite[Theorem 3.6]{AlarconDrinovecForstnericLopez2015PLMS}
or \cite[Theorem 6.6.2]{AlarconForstnericLopez2021}), 
performing the construction simultaneously on the $G$-orbits of the 
boundary arcs. We explain the modifications. 

By Remark \ref{rem:XG}, we may assume that $X$ is a relatively compact 
$G$-invariant domain with real analytic boundary in an open 
Riemann surface $R$, that $G$ acts on $R$ by holomorphic 
automorphisms, and its stabiliser locus is a finite set $X_0\subset X$.
Choose a $G$-invariant holomorphic $1$-form $\theta$ on $R$ 
as in \eqref{eq:theta}. Thus, 
$\{\theta=0\}=X_0$. Let $Y=\bA_*\cup Y_0\subset \CP^n$ be 
the manifold \eqref{eq:Y}. The map  
\begin{equation}\label{eq:f}
	f=2\di F/\theta :\overline X\to Y
\end{equation}
is of class $\Ascr(\overline X,Y)$, holomorphic on $X$, and $G$-equivariant. 
Since the action of $G$ is free in a neighbourhood of $bX$, after shrinking 
the annulus $A$ around $bX$ we have $A\cap X_0=\varnothing$ and hence
\[ 
	f(A)\subset\bA_*. 
\] 
Let $I_{i_1},\ldots,I_{i_s}$ be representatives of the $G$-orbits in 
the set of arcs $\{I_1,\ldots,I_k\}$. Choose pairwise disjoint closed 
boundary discs 
\[ 
	D_{i_\nu}\subset U\cap A,\qquad \nu=1,\ldots,s
\] 
with piecewise smooth boundaries such that for every $\nu=1,\ldots,s$ the arc
\[ 
	I_{i_\nu}\subset bX\cap bD_{i_\nu} 
\] 
lies in the relative interior of $bX\cap bD_{i_\nu}$ and  
\[ 
	\rho(D_{i_\nu})\subset bX\cap D_{i_\nu}. 
\]
These discs may be chosen so small that the sets 
\[ 
	D_{g,i_\nu}:=g(D_{i_\nu}), \qquad g\in G,\quad \nu=1,\ldots,s, 
\] 
are pairwise disjoint. If $g(I_{i_\nu})=I_j$, we write $D_j=g(D_{i_\nu})$. 
In this way we obtain a $G$-invariant family of pairwise disjoint discs 
$D_1,\ldots,D_k$. 

Let $\sigma$ be as in the theorem. 
On every representative disc $D_{i_\nu}$ we apply the local 
Riemann--Hilbert construction used in the proof of 
\cite[Theorem 3.6]{AlarconDrinovecForstnericLopez2015PLMS}
(see also \cite[Theorem 6.6.2]{AlarconForstnericLopez2021})
with the null direction 
\[ 
	\bw_{i_\nu}=\bu_{i_\nu} - \imath\bv_{i_\nu} \in\bA_*
\]
and the attached family of conformal minimal discs 
\[ 
	\cD\ni \xi\longmapsto 
	\varkappa(x,\xi) = 
	F(x)+\Re \sigma(x,r(x)\xi) \bu_i+ \Im \sigma(x,r(x)\xi)  \bv_i, 
	\quad x\in I_{i_\nu},\ \xi \in \cD 
\] 
(see \eqref{eq:kappa}). This yields holomorphic null data 
\[
	f_{i_\nu}:D_{i_\nu}\times \cD \to Y,\quad \nu=1,\ldots, s
\] 
on the representative discs $D_{i_\nu}$ such that for some point
$p_{i_\nu}\in \mathring D_{i_\nu}$, the integrals 
\[
	F(p_{i_\nu}) + \Re \int_{p_{i_\nu}}^x f_{i_\nu}(x,\xi) \theta,
	\qquad  x\in D_{i_\nu},\ \xi\in \cD 
\]
satisfy the estimates (i)--(iii). 
(The integral from $p_{i_\nu}$ to $x$ is taken within the disc 
$D_{i_\nu}$. See the proofs of Theorems 6.6.1 and 6.6.2 in 
\cite{AlarconForstnericLopez2021} for the details.)
We extend this construction to every disc in the orbit by $G$-equivariance. 
More precisely, for any $g\in G$ we define the null data 
on the disc $D_j=g(D_{i_\nu})$ by 
\begin{equation}\label{eq:fj}
	f_j(x,\xi)=g f_{i_\nu}(g^{-1}x,\xi),\qquad x\in D_j,\ \xi\in \cD. 
\end{equation} 
The definition is unambiguous since the action of $G$ on the set of arcs, 
and hence on the chosen family of discs, is free. The identities 
\[ 
	\sigma(gx,\xi)=\sigma(x,\xi),\qquad r(gx)=r(x), 
	\qquad g\bw_{i_\nu}=\bw_j 
\] 
show that the resulting integrals 
satisfy the same Riemann--Hilbert estimates on all the discs $D_j$ for 
$j=1,\ldots,k$, and the family is $G$-equivariant. 

It remains to glue the local null data $f_j$ \eqref{eq:fj} to $f$ \eqref{eq:f}
so that the integral of the new $G$-equivariant holomorphic map 
$\tilde f:\overline X\to Y$ 
is well-defined and satisfies the required interpolation, approximation and 
period conditions. For this purpose, we use a 
$G$-equivariant period-dominating spray with core $f$. The construction 
of such sprays, including the periods needed for equivariant integration, 
is given in \cite[Lemma 4.5 and Sect.~5]{Forstneric2024PLMS}. 
We choose the spray so that its members agree with $f$ to order $d$ at 
every point of the $G$-invariant finite set $\Lambda$.
In addition, the orders are chosen to preserve the pole orders and 
leading coefficients of $f$ at every point of $X_0$, so that $\tilde f\theta$ 
extends holomorphically and is nowhere zero at these points.
Likewise, we construct a $G$-equivariant family of sprays on the discs 
$D_j$, with core $f_j$ \eqref{eq:fj}, such that these sprays approximate 
the one over $f$ as closely as desired 
outside a small neighbourhood of the arc $I_j$.
To complete the proof, we glue the spray 
with the core $f$ and the family of sprays over the discs $D_j$ with the 
core $f_j$ into a new $G$-equivariant holomorphic spray on $\overline X$  
with values in $Y$ \eqref{eq:Y} such that a certain member
\[ 
	\tilde f:\overline X\longrightarrow Y 
\] 
approximates $f$ in the $\Cscr^0$ 
topology off a small neighbourhood $\Omega\subset U\cap A$ 
of $I$, and the integral $\wt F:\overline X\to \R^n$ of $\tilde f$ 
is well defined, $G$-equivariant, and satisfies conditions (i)--(v)
in the theorem. Furthermore, the jet interpolation conditions 
on $\tilde f$ at the points of $X_0$ ensure that 
\[ 
	\tilde f(\overline X\setminus X_0)\subset\bA_*, 
	\qquad \tilde f(X_0)\subset Y_0, 
\] 
so that $\tilde f\theta$ extends holomorphically and is nowhere zero.
This ensures that $\wt F$ is an immersion.

For gluing of sprays, and to find a suitable member $\tilde f$ of the 
resulting glued spray, we use the correspondence in 
Lemma \ref{lem:correspondence} as in 
\cite[proof of Theorem 5.1]{Forstneric2024PLMS}.
Note that $G$-equivariant maps (and sprays of such maps) 
from $G$-invariant subdomains of $\overline X$ to $Y$  
pass down to the quotient $\overline X/G\subset R/G$
as sections (or sprays of sections) of the holomorphic fibration
$\rho:Z\to R/G$ \eqref{eq:Z}; conversely, sections 
of $\rho$ yield $G$-equivariant maps to $Y$. 
Hence, we can perform the gluing 
of sprays over the base $\overline X/G$ by using  
\cite[proof of Theorem 3.5]{AlarconDrinovecForstnericLopez2015PLMS}, 
coupled with the technique from \cite[Theorem 2.1]{Forstneric2003FM}
(see also \cite[Theorem 6.14.6]{Forstneric2017E}), 
which yields the Oka principle for sections of certain ramified 
holomorphic maps onto a Stein base.  
Fix a point $x_0\in X$ and for every $a\in\Lambda$ 
choose a path $\gamma_a$ in $X$ from $x_0$ to $a$.
As shown in \cite[proof of Lemma 4.5]{Forstneric2024PLMS},
the parameter in the dominating $G$-equivariant spray of maps 
$\overline X\to Y$, obtained by gluing, 
can be chosen so that the corresponding member $\tilde f$ 
of the spray satisfies
\begin{equation}\label{eq:RHperiods} 
	\int_C(\tilde f-f)\theta=0 \quad\text{for every }[C]\in H_1(X,\Z), 
\end{equation} 
\begin{equation}\label{eq:RHperiods2}
	\Re\int_{x_0}^{gx_0}(\tilde f-f)\theta=0, \qquad g\in G, 
\end{equation} 
\begin{equation}\label{eq:RHperiods3}
	\Re\int_{\gamma_a}(\tilde f-f)\theta=0,\qquad a\in\Lambda.
\end{equation} 
The real parts of equations \eqref{eq:RHperiods} ensure that the map 
$\wt F:\overline X\to\R^n$, given by 
\[ 	
	\widetilde F(x) 
	= F(x_0)+\Re\int_{x_0}^x\tilde f\theta, \qquad x\in\overline X, 
\] 
is a well-defined conformal minimal immersion,
whereas the imaginary parts of  \eqref{eq:RHperiods} ensure that 
$\Flux_{\widetilde F}=\Flux_F$ (condition (v)). 
Condition \eqref{eq:RHperiods2} implies 
that \eqref{eq:gv} remains unchanged: 
\begin{equation}\label{eq:gv2}
	g\wt F(x_0)-\wt F(x_0) =
	 \Re\int_{x_0}^{gx_0}\tilde f\,\theta, \qquad g\in G.
\end{equation}
The $G$-equivariance of $\tilde f$ and \eqref{eq:gv2} imply 
exactly as in \eqref{eq:ggamma} that $\widetilde F$ is $G$-equivariant. 
Condition \eqref{eq:RHperiods3} implies that $\widetilde F-F$ 
vanishes at every point of $\Lambda$. This and  the jet-interpolation 
conditions of sufficiently high order, imposed on the spray at the points 
of $\Lambda$, imply that $\widetilde F-F$ vanishes to order $d$ 
at every point of $\Lambda$, so (iv) holds. 
The estimates in the local Riemann--Hilbert 
construction are preserved up to a multiplicative constant in 
the gluing process, and the gluing can be made arbitrarily 
close to the original spray on $\overline X\setminus\Omega$,
so we can also ensure that conditions (i)--(iii) hold. 
\end{proof}
%
%
%

%
%
\begin{corollary}\label{cor:RH-null} 
The analogue of Theorem \ref{th:RH} holds for $G$-equivariant 
holomorphic null immersions into $\C^n$ for any $n\ge 3$, where 
$G$ acts on $\C^n$ by elements of the complex orthogonal group 
$O(n,\C)$. The approximation and jet-interpolation conclusions remain 
valid, and the deformation may be chosen so that the null differential 
has the same complex periods as the initial one. 
In particular, exact null differentials remain exact. 
\end{corollary} 
\begin{proof} The proof is the same as that of Theorem \ref{th:RH}, except that one does not take real parts when integrating the null differential. The full complex period map is used in the period-correction step. 
\end{proof}

%
%
%
%
\section{Main lemmas}\label{sec:lemmas}
In this section, we state and prove the main lemmas concerning properness, completeness, and general position that will lead to the proof of the main theorems via standard recursive procedures. The proofs of these lemmas rely on the new tools established in the previous section.

%
%
\subsection{$G$-equivariant properness lemma}\label{sec:properL}
Our proof of Theorem \ref{th:proper} is an adaptation to the 
$G$-equivariant case of the boundary-lifting method, used in the 
construction of proper conformal minimal immersions into convex domains 
in $\R^n$, $n\ge 3$, and in particular to $\R^n$ itself; see 
\cite[Theorem 1.4]{AlarconDrinovecForstnericLopez2015PLMS}.
(Stronger results concerning proper conformal minimal immersions 
into minimally convex domains in $\R^3$ and $(n-2)$-convex 
domains in $\R^n$, $n\ge3$, can be found in 
\cite{AlarconDrinovecForstnericLopez2019TAMS}
and \cite[Sect.~8.4]{AlarconForstnericLopez2021}.)
The proofs of these properness results inductively use
the boundary lifting method furnished by exposing boundary points and the Riemann-Hilbert deformation techniques described in the
previous section. The induction is based on the following boundary lifting 
lemma, whose proof relies on Corollary \ref{cor:curves} and Theorem \ref{th:RH}. The lemma will be also used in the proof of Theorem \ref{th:complete-proper}.

%
%
\begin{lemma}\label{lem:proper}
Let $X$ and $G$ be as in Theorem \ref{th:proper}, let 
$F\in\CMI_{G,*}^1(\overline X,\R^n)$, and assume that for a compact 
$G$-invariant set $K\subset \mathring X$ and numbers $0<r<s$ we have
\begin{equation}\label{eq:lemproper}
	F(\overline {X\setminus K}) \subset s\B \setminus r\overline\B,
\end{equation}
where $\B\subset \R^n$ is the unit ball.
Given a number $\delta$ satisfying $0<\delta<s-r$, 
$F$ can be approximated as closely as desired in $\Cscr^1(K,\R^n)$ by  
$\wt F\in\CMI_{G,*}^1(\overline X,\R^n)$ satisfying the following properties.
\begin{enumerate}[\rm (a)]
\item $\|\wt F-F\|_{0,\overline X} < \sqrt{s^2-r^2}$.
\item $\wt F(\overline {X\setminus K}) \subset s\B \setminus r\overline\B$.
\item $\wt F(bX) \subset s\B \setminus (s-\delta)\overline\B$.
\item $\Flux(\wt F)=\Flux(F)$. 
\end{enumerate}
\end{lemma}

The main point in the lemma is condition (c), which asserts that $\wt F(bX)$ 
lies near the boundary of the bigger ball $s\B$ when $\delta>0$ is small.

%
%
\begin{proof}
By Remark \ref{rem:XG}, we may assume that $X$ is a relatively compact 
$G$-invariant domain with real analytic boundary in an open Riemann surface 
$R$, the group $G$ acts on $R$ by holomorphic automorphisms, 
and the action is free on $R\setminus X$. By \eqref{eq:lemproper} 
and compactness of $bX$, we may also assume by decreasing 
$\delta>0$ if necessary that $r+2\delta<s$ and
\begin{equation}\label{eq:lemproper2}
	F(\overline {X\setminus K}) \subset s\B\setminus (r+\delta)\overline \B.
\end{equation}
Denote by $\langle\cdot,\cdot\rangle$ the Euclidean scalar product in $\R^n$. 
Fix a number $\mu$ with $0<\mu<\delta$ so small that the following holds: 
\begin{equation}\label{eq:mu}
	\text{If \; $p,q\in s\overline \B\setminus (r+\delta)\B$ \; and  
	\;$\langle p-q,q\rangle>-\mu$, \; then \; $|p-q|<\sqrt{s^2-r^2}$.}
\end{equation}
Such $\mu$ exists by Pythagoras' theorem. To prove the lemma, 
we shall provide $\wt F\in\CMI_{G,*}^1(\overline X,\R^n)$, 
as close to $F$ as desired in $\Cscr^1(K,\R^n)$, satisfying conditions 
(b), (c), (d), and
\begin{equation}\label{eq:mu2}
	\langle \wt F(x)-F(x),F(x)\rangle>-\mu\quad \text{for all }x\in bX.
\end{equation}
This, \eqref{eq:lemproper2}, and \eqref{eq:mu} imply that 
$|\wt F(x)-F(x)|<\sqrt{s^2-r^2}$ for all $x\in bX$, which ensures (a) 
by the maximum principle. We will complete the task in two deformation steps.
 
\noindent{\bf First deformation.} Fix a number $\epsilon_0>0$ to be specified later. 
Let $\Jcal$ be a finite collection of closed arcs in $bX$ satisfying the following 
conditions.
\begin{enumerate}[(i)]
\item $bX=\bigcup_{I\in \Jcal} I$ and each connected component of $bX$ 
contains at least three arcs in $\Jcal$.
\item $I\cap I'$ is either empty or a singleton for every pair of arcs $I\neq I'\in\Jcal$. 
\item $g(I)\cap I=\varnothing$ and $g(I)\in \Jcal$ for every $I\in \Jcal$ and 
$g\in G\setminus\{1\}$. In particular, $\Jcal$ is $G$-invariant.
\item If $I\in\Jcal$ and $x,y\in I$, then $|F(x)-F(y)|<\epsilon_0$.
\end{enumerate}
Such a collection exists by compactness of $bX$ and continuity of $F$.
It follows that if $x\in I\cap I'$ with $I\neq I'\in\Jcal$, then $x$ is a common 
endpoint of $I$ and $I'$. Moreover, each arc in $\Jcal$ intersects exactly two 
other arcs in $\Jcal$. Denote by $P\subset bX$ the finite set formed by the 
endpoints of the arcs in $\Jcal$, that is, 
\[
	P=\{x\in bX : \exists I\neq I'\in\Jcal,\; I\cap I'= 
	\{x\}\}.
\] 
Write $P=\{x_1,\ldots, x_m\}$ and note that $P$ is a $G$-invariant set. 
Let $E_1,\ldots,E_m$ be smoothly embedded pairwise disjoint compact arcs 
in $R$ such that for each $i=1,\ldots,m$, $E_i$ is attached to $\overline X$ 
transversely at $x_i$, $E_i\cap\overline X=\{x_i\}$, and the set 
$E=\bigcup_{i=1}^m E_i$ is $G$-invariant. Let $y_i$ denote the other endpoint 
of $E_i$, set $Q=\{y_1,\ldots,y_m\}$, and consider the $G$-invariant 
admissible set
\[
	S=\overline X\cup E\subset R.
\]
For each $i\in\{1,\ldots,m\}$, choose a number $\lambda_i>1$ such that 
\begin{equation}\label{eq:lambdai}
	s-\delta<\lambda_i|F(x_i)|<s
	\quad\text{and}\quad \lambda_i|F(x_i)|=\lambda_j|F(x_j)|\quad \text{for all }i,j.
\end{equation}
Next, given $\epsilon_1>0$ to be specified later, we extend $F$, with the same 
name, to a nonflat $G$-equivariant generalised conformal minimal immersion 
$F\in \GCMI^1_{G,*}(S,\R^n)$ such that for every $i\in\{1,\ldots,m\}$,
\begin{equation}\label{eq:Eextension}
	d^H\big(F(E_i),[F(x_i),\lambda_iF(x_i)]\big)<\epsilon_1
	\quad\text{and}\quad
	F(y_i)=\lambda_iF(x_i),
\end{equation}
where $d^H$ denotes the Hausdorff distance in $\R^n$ and $[p,q]$ the straight 
segment connecting a pair of points $p,q\in\R^n$. As usual, we first extend $F$ 
to a representative arc $E_i\subset E$ in each family (this is possible by 
\cite[Lemma 3.5.4]{AlarconForstnericLopez2021}), and then we extend $F$ 
to every translate $g(E_i)$ of $E_i$ by equivariance: $F(gx)=gF(x)$ for all 
$x\in E_i$. Note that 
\[
	g\big([F(x_i),\lambda_iF(x_i)]\big)=[gF(x_i),\lambda_igF(x_i)]
	\quad \text{ for every $g\in G$}
\]
and, by \eqref{eq:lemproper2},
\begin{equation}\label{eq:[]}
	[F(x_i),\lambda_iF(x_i)]\subset s\B\setminus (r+\delta)\overline\B.
\end{equation}
In view of \eqref{eq:lemproper2}, \eqref{eq:lambdai}, and \eqref{eq:Eextension}, 
and assuming that $\epsilon_1>0$ is sufficiently small, we have that 
\begin{equation}\label{eq:>-mu/2}
	F(E)\subset s\B\setminus (r+\delta)\overline\B  	
	\quad\text{ and }\quad
	\langle F(x)-F(x_i),F(x_i)\rangle>-\frac{\mu}2\quad 
	\text{for all }x\in E_i,\; i=1,\ldots,m.
\end{equation}
Given a number $\epsilon_2>0$, Corollary \ref{cor:curves} then provides a 
nonflat $G$-equivariant conformal minimal immersion $F':\overline X\to\R^n$ 
and a $G$-invariant collection of (big) closed subarcs 
$\{J_I\subset \mathring I: I\in \Jcal\}$ satisfying the following.
\begin{enumerate}[\rm (AA)]
\item[(v)] $F'$ is as close to $F$ as desired in $\Cscr^1(K,\R^n)$.
\item[(vi)] $F'(x_i)=\lambda_iF(x_i)\in s\B\setminus (s-\delta)\overline\B$ 
for every $i\in\{1,\ldots,m\}$; see \eqref{eq:lambdai} and \eqref{eq:Eextension}.
\item[(vii)] $|F'(x)-F(x)|<\epsilon_2$ for all $x\in \bigcup_{I\in\Jcal}J_I$.
\item[(viii)] $\langle F'(x)-F(x),F(x)\rangle>-\mu/2$ for all 
$x\in bX\setminus  \bigcup_{I\in\Jcal}J_I$; see \eqref{eq:>-mu/2}
\item[(ix)] $d^H\big(F'(\Gamma_i),[F(x_i),\lambda_iF(x_i)]\big)<\epsilon_1$ 
for every $i\in\{1,\ldots,m\}$, where $\Gamma_i\subset bX$ is the closure of 
the connected component of $bX\setminus  \bigcup_{I\in\Jcal} J_I$ 
that contains the point $x_i$.
\item[(x)] $F'(\overline X\setminus K)\subset s\B\setminus (r+\delta)\overline\B$; 
see \eqref{eq:lemproper2} and \eqref{eq:>-mu/2}.
\item[(xi)] $\Flux_{F'}=\Flux_F$.
\end{enumerate}
Assuming that $\epsilon_2>0$ is sufficiently small, we infer from (vii) and (viii) that
\begin{equation}\label{eq:F'>-mu/2}
	\langle F'(x)-F(x),F(x)\rangle>-\frac{\mu}2\quad \text{for all }x\in bX.
\end{equation}
On the other hand, condition (vi) guarantees the existence of a $G$-equivariant 
collection of (bigger) closed subarcs $\{J_I'\subset \mathring I\colon I\in\Jcal\}$ 
such that $J_I\subset \mathring J_I'$ for every $I\in\Jcal$ and
\begin{equation}\label{eq:JI'}
	F'(x)\in s\B\setminus (s-\delta)\overline\B\quad \text{for all }
	x\in bX\setminus\bigcup_{I\in\Jcal}\mathring J_I'.
\end{equation}
Therefore, $F'\in\CMI_{G,*}^1(\overline X,\R^n)$ formally satisfies conditions 
(b), (d), and \eqref{eq:mu2}, but it only meets (c) at the boundary points outside 
$\bigcup_{I\in\Jcal} \mathring J_I'$.

\noindent{\bf Second deformation.} Set $\Upsilon=\bigcup_{I\in\Jcal}J_I$ and 
$\Upsilon'=\bigcup_{I\in\Jcal}J_I'$, and denote by $\Gamma_i'\subset bX$ the 
closure of the connected component of 
$bX\setminus  \bigcup_{I\in\Jcal}\mathring J_I'$ that contains the point $x_i$ 
(hence, $\Gamma_i'\subset\mathring\Gamma_i$). For each $I\in\Jcal$ fix a point 
$z_I\in\mathring I$ making sure that the set $\{z_I\colon I\in\Jcal\}$ is $G$-invariant.
Assuming that the fixed positive numbers $\epsilon_0$, $\epsilon_1$, and 
$\epsilon_2$ are sufficiently small, for every $I\in\Jcal$ there exist 
\begin{itemize}
\item a pair of orthogonal unit vectors $(\bu_I,\bv_I)$ in $\R^n$, that are 
orthogonal to $F'(z_I)\neq 0$,  such that if for some $g\in G$ we have $g(I)=I'$ then 
$g\bu_I=\bu_{I'}$ and $g\bv_I=\bv_{I'}$,
\item a continuous $G$-equivariant function $\tau:bX\to[0,1]$ supported on 
$\Upsilon'$, and
\item a $\Cscr^1$ function $\sigma:\Upsilon'\times\overline\D\to \C$ which is 
$G$-invariant in the first variable, for every $x\in\Upsilon'$ the function 
$\sigma(x,\cdot)$ is holomorphic in $\D$, $\sigma(x,0)=0$, and 
$\frac{\partial\sigma}{\partial\xi}$ is nowhere vanishing on 
$\Upsilon'\times\overline\D$, where $\xi$ is the second variable,
\end{itemize}
such that the continuous map
$\varkappa:bX \times\overline{\D}\to\R^n$ given by
\begin{equation}\label{eq:varkappa-lem}
	\varkappa(x,\xi)=
	\begin{cases}
		F'(x),  & \text{if}\ x \in bX\setminus \Upsilon'; \\
		F'(x) + \Re \sigma(x,\tau(x)\xi) \bu_I+ \Im \sigma(x,\tau(x)\xi)  \bv_I, &
		\text{if}\ x \in J_I',\; I\in\Jcal
	\end{cases}
\end{equation}
satisfies the following conditions.
\begin{enumerate}[(I)]
\item $\varkappa(x,\xi)\in s\B\setminus(s-\delta)\overline\B$ for all $x\in bX$, 
$\xi\in\T$.
\item $\varkappa(bX\times\overline\D)\subset s\B\setminus(r+\delta)\overline\B$. 
\end{enumerate}
Indeed, take into account \eqref{eq:[]}, \eqref{eq:JI'}, and that choosing 
$\epsilon_0$, $\epsilon_1$, and $\epsilon_2$ small enough, if $x_i$ is an endpoint 
of $I\in\Jcal$ then $F'(I)$ is as close as desired to $[F(x_i),\lambda_iF(x_i)]$ in Hausdorff distance; see conditions (iv), (vii), and (ix). 
(Note that if these numbers are small and $x_j$ is the other endpoint of $I$, 
then $[F(x_i),\lambda_iF(x_i)]$ and $[F(x_j),\lambda_jF(x_j)]$ are close in Hausdorff 
distance by the aforementioned conditions and the second property in 
\eqref{eq:lambdai}). As above, we first define the objects on a representative arc 
$I\in\Jcal$ in each family, and then on every translate $g(I)$ of $I$ by equivariance.

In this situation, Theorem \ref{th:RH} provides a nonflat $G$-equivariant conformal 
minimal immersion $\wt F:\overline X\to\R^n$ satisfying the following conditions.
\begin{enumerate}[(AA)]
\item[(III)] $\wt F$ is as close as desired to $F'$ in $\Cscr^1(K,\R^n)$.
\item[(IV)] $\wt F(bX)\subset s\B\setminus(s-\delta)\overline\B$; see (I). 
\item[(V)] $\langle \wt F(x)-F(x),F(x)\rangle>-\mu$ for all $x\in bX$; see 
\eqref{eq:F'>-mu/2} and (iv), recall that $\bu_I$ and $\bv_I$ are orthogonal to 
$F'(z_I)$ for every $I\in\Jcal$, and assume that $\epsilon_0>0$ is sufficiently small.
\item[(VI)] $\wt F(\overline X\setminus K)\subset s\B\setminus r\overline\B$; 
see (x) and (II).
\item[(VII)] $\Flux_{\wt F}=\Flux_{F'}$.
\end{enumerate}
It follows that $\wt F$ is as close as desired to $F$ in $\Cscr^1(K,\R^n)$ 
(see (v) and (III)) and it satisfies conditions (b), (c), (d) (see (xi) and (VII)), and 
\eqref{eq:mu2}, thereby proving the lemma.
\end{proof}

%
%
\subsection{$G$-equivariant completeness lemma}\label{sec:completenessL}

Given a conformal minimal immersion $F:\overline X\to\R^n$, we denote by 
$\dist_F$ the distance function on $\overline X$ induced by the Riemannian 
metric $F^*ds^2$. If $A,B\subset\overline X$, we write 
\[ 
	\dist_F(A,B) = \inf\{\dist_F(x,y):x\in A,\ y\in B\}. 
\] 
The main ingredient in the proof of Theorems \ref{th:complete} and 
\ref{th:complete-proper} (together with Lemma \ref{lem:proper} for the latter) 
is the following equivariant 
version of \cite[Lemma 4.1]{AlarconDrinovecForstnericLopez2015PLMS}. 
It asserts that every $G$-equivariant conformal minimal immersion 
$\overline X\to\R^n$ of a compact bordered Riemann surface
can be approximated as closely as desired in $\Cscr^0(\overline X)$ by 
$G$-equivariant conformal minimal immersions 
with arbitrarily large intrinsic diameter. 

%
%
\begin{lemma}\label{lem:equivariant-distance} 
Let $X$, $G$, and $n$ satisfy the hypotheses of Theorem \ref{th:complete}, 
and let $F:\overline X\to\R^n$ be a nonflat $G$-equivariant conformal minimal 
immersion of class $\Cscr^1(\overline X)$. Fix a point $x_0\in X$, a compact 
$G$-invariant set $K\subset X$, and numbers $\lambda>0$ (large) and 
$\epsilon>0$ (small).  
Then there exists a nonflat $G$-equivariant conformal minimal 
immersion $\widetilde F:\overline X\to\R^n$ of class 
$\Cscr^1(\overline X)$ satisfying the following conditions. 
\begin{enumerate}[\rm (a)] 
\item $\|\widetilde F-F\|_{\Cscr^0(\overline X)}<\epsilon$.
\item $\|\widetilde F-F\|_{\Cscr^1(K)}<\epsilon$.
\item $\dist_{\widetilde F}(x_0,bX)>\lambda$.
\item $\Flux_{\widetilde F}=\Flux_F$.
\end{enumerate} 
\end{lemma} 

The main condition is (c), which shows that the intrinsic distance from an interior point to the boundary of $X$ in the metric induced by $\wt F$ can be made as 
big as desired. The proof of Lemma \ref{lem:equivariant-distance} follows by a 
finite recursive application of the technical lemma below, in exactly the same way 
that \cite[Lemma 4.1]{AlarconDrinovecForstnericLopez2015PLMS} follows from 
\cite[Lemma 4.2]{AlarconDrinovecForstnericLopez2015PLMS}; hence, we omit 
the details.

%
%
\begin{lemma}\label{lem:equivariant-distance-2}
Let $X$, $G$, $n$, $F$, $x_0$, and $K$ be as in 
Lemma \ref{lem:equivariant-distance}. 
Let $\Phi:bX\to\R^n$ be a continuous $G$-equivariant 
map and $\delta>0$ and $d>0$ be numbers such that
\begin{equation}\label{eq:lemcomplete}
	\|F-\Phi\|_{0,bX}<\delta
	\quad\text{ and }\quad 0<d<\dist_F(x_0,bX).
\end{equation}
Given $\eta>0$, $F$ can be approximated as closely as desired in 
$\Cscr^1(K,\R^n)$ by $\wt F\in\CMI_{G,*}^1(\overline X,\R^n)$ satisfying the 
following properties.
\begin{enumerate}[\rm (a)]
\item $\|\wt F-\Phi\|_{0,bX} < \sqrt{\delta^2+\eta^2}$.
\item $\dist_{\wt F}(x_0,bX)>d+\eta$.
\item $\Flux(\wt F)=\Flux(F)$. 
\end{enumerate}
\end{lemma} 

%
%
\begin{proof}
The proof consists of adapting that of 
\cite[Lemma 4.2]{AlarconDrinovecForstnericLopez2015PLMS} to the 
equivariant case by using similar ideas as those in the proof of 
Lemma \ref{lem:proper}. 
The proof follows a similar scheme to that of the proof of Lemma 
\ref{lem:proper}, consisting of two deformation steps. We proceed 
to explain it.

First of all, by Remark \ref{rem:XG}, we may assume that $X$ is a relatively compact 
$G$-invariant domain with real analytic boundary in an open Riemann surface $R$, 
that $G$ acts on $R$ by holomorphic automorphisms, and that $G$ acts freely on 
$R\setminus X$. We may also assume by a slight perturbation of $\Phi$ that
\begin{equation}\label{eq:F-Phineq0}
	\text{$F-\Phi$ vanishes nowhere on $bX$}.
\end{equation}
Finally, we can also assume without loss of generality that $K$ is a smoothly 
bounded compact domain that is a strong deformation retract of $X$ with $x_0\in\mathring K$ and, in view of  
\eqref{eq:lemcomplete}, there is a number $\mu_0>0$ (small) such that
\begin{equation}\label{eq:distx0bK}
	\|F-\Phi\|_{0,bX}<\delta-\mu_0
	\quad\text{ and }\quad \dist_F(x_0,bK)>d+\mu_0.
\end{equation}

\noindent{\bf First deformation.} 
Fix $\epsilon_0>0$ to be specified later 
and choose a collection $\Jcal$ of closed arcs in $bX$ satisfying conditions 
(i)--(iv) in the proof of Lemma \ref{lem:proper}. In addition, we take the arcs 
so small that
\begin{equation}\label{eq:Phie0}
	|\Phi(x)-\Phi(y)|<\epsilon_0\quad \text{ for all }x,y\in I,\; I\in\Jcal.
\end{equation}
We define 
\[
	P=\{x_1,\ldots,x_m\},\quad 
	E=\bigcup_{i=1}^mE_i,
	\quad Q=\{y_1,\ldots,y_m\},
	\quad S=\overline X\cup E\subset R
\]
as in that proof. 
Arguing as in the first deformation step in the proof of Lemma \ref{lem:proper}, 
we extend $F$ to a nonflat $G$-equivariant generalised conformal minimal 
immersion $F\in \GCMI^1_{G,*}(S,\R^n)$ such that for every 
$i\in\{1,\ldots,m\}$, the curve $F(E_i)$ has small extrinsic diameter 
(that is, it is contained in a small ball in $\R^n$ centred at $F(x_i)$)
but its projection onto the vector $F(x_i)-\Phi(x_i)\neq 0$
(see \eqref{eq:F-Phineq0}) has total variation bigger than $\eta$. For this, 
we choose the extension of $F$ to $E_i$ to be highly oscillating while 
keeping $F(E_i)$ very close in Hausdorff distance to the straight line segment 
\[
	\big[F(x_i),F(x_i)+\lambda(F(x_i)-\Phi(x_i))/|F(x_i)-\Phi(x_i)|\big]
\] 
for a small $\lambda>0$.
Next, Corollary \ref{cor:curves} furnishes a nonflat $G$-equivariant conformal 
minimal immersion $F':\overline X\to\R^n$ and a $G$-equivariant collection 
of (big) closed subarcs $\{J_I\subset\mathring I\colon I\in\Jcal\}$ satisfying 
the following conditions (cf.\ properties (v)--(xi) in the proof of 
Lemma \ref{lem:proper} as well as 
\cite[Lemma 4.3]{AlarconDrinovecForstnericLopez2015PLMS}):
\begin{enumerate}[\rm (AA)]
\item[(v)] $F'$ is as close to $F$ as desired in $\Cscr^1(K,\R^n)$.
\item[(vi)] $F'$ is close to $F$ uniformly in $\overline X$ 
(depending on how small the extrinsic diameter of the paths $F(E_i)$ is
for $i=1,\ldots,m$). In particular, 
we may guarantee in view of (iv) that if $I\in\Jcal$ and $x,y\in I$ then 
\[
	\max\big\{|F(x)-F(y)|,|F'(x)-F'(y)|,|F'(x)-F(y)|\big\}<\epsilon_0.
\]
\item[(vii)] Let $\Upsilon=\bigcup_{I\in\Jcal}J_I\subset bX$ and for every 
$i\in\{1,\ldots,m\}$ denote by $\Gamma_i\subset bX$ the closure of the 
connected component of $bX\setminus  \bigcup_{I\in\Jcal} J_I$ 
containing $x_i$. For every $i=1,\ldots, m$ there is a closed 
smoothly bounded neighbourhood $W_i$ of $\Gamma_i$ in $\overline X$, 
disjoint from $K$, such that the sets $W_1,\ldots, W_m$ 
are pairwise disjoint, and for every path $\gamma\subset \overline X$ 
connecting $bW_i\cap X$ to $\Gamma_i$ the projection of the path 
$F'(\gamma)\subset\R^n$ onto the vector $F(x_i)-\Phi(x_i)\neq 0$ 
has total variation bigger than $\eta$. 
\item[(viii)] $\Flux_{F'}=\Flux_F$.
\end{enumerate}

In constructing $F'$, we impose in part (vii) the partition-based 
projected-length condition of 
\cite[Lemma 4.3, property (P4)]{AlarconDrinovecForstnericLopez2015PLMS}. 
The construction is performed on representative attached arcs and extended 
to their translates by $G$-equivariance.

\noindent{\bf Second deformation.} 
Take a $G$-equivariant collection of closed subarcs 
$\{J_I'\subset \mathring I\colon I\in\Jcal\}$ such that 
$J_I\subset \mathring J_I'$ for every $I\in\Jcal$ and set 
$\Upsilon'=\bigcup_{I\in\Jcal}J_I'$. Choose points $z_I\in\mathring I$, 
$I\in\Jcal$, as in the proof of Lemma \ref{lem:proper}. 
For every $I\in\Jcal$ take an orthonormal frame $(\bu_I,\bv_I)$ 
in $\R^n$ orthogonal to $F(z_I)-\Phi(z_I)\neq 0$ 
(see \eqref{eq:F-Phineq0}), a continuous 
$G$-equivariant map $\tau:bX\to[0,1]$ supported on $\Upsilon'$, 
and a $\Cscr^1$ function $\sigma:\Upsilon'\times\overline\D\to\C$ 
as in the second deformation step in the proof of Lemma \ref{lem:proper}, 
such that
\begin{equation}\label{eq:leeta}
	|\sigma(x,\tau(x)\xi)|\le\eta\ \ \text{for all }(x,\xi)\in bX\times\T
	\quad\text{and}\quad
	|\sigma(x,\tau(x)\xi)|=\eta\ \ \text{for all }(x,\xi)\in \Upsilon\times\T.
\end{equation}
Given $\epsilon_1>0$ to be specified later, Theorem \ref{th:RH} applied 
with these data and the map $\varkappa:bX \times\overline{\D}\to\R^n$ 
given by \eqref{eq:varkappa-lem} provides a nonflat $G$-equivariant 
conformal minimal immersion $\wt F:\overline X\to\R^n$ satisfying the 
following conditions.
\begin{enumerate}[(AA)]
\item[(I)] $\wt F$ is as close as desired to $F'$ in 
$\Cscr^1(\overline X\setminus\Omega,\R^n)$, where 
$\Omega\subset \overline X$ 
is an arbitrarily small neighbourhood of $\Upsilon'$ with 
$K\cap\Omega=\varnothing$. In particular, we can ensure in view of (vi) 
that if $x\in I\in \Jcal$ and $y\in X\cap b\Omega$ lies in the component of 
$b\Omega$ intersecting $I$, then $|\wt F(y)-F'(x)|<\epsilon_0$.
\item[(II)] $\eta-\epsilon_1<|\wt F(x)-F'(x)|< \eta+\epsilon_1$ for all 
$x\in \Upsilon$, $|\wt F(x)-F'(x)|< \eta+\epsilon_1$ for all $x\in \Upsilon'$, 
and $|\wt F(x)-F'(x)|< \epsilon_1$ for all $x\in bX\setminus\Upsilon'$. 
See \eqref{eq:leeta} and \eqref{eq:varkappa-lem}, and recall that the 
directing vectors $\bu_I$ and $\bv_I$ are unit vectors for every $I\in\Jcal$.
\item[(III)] $|\langle \wt F(x)-F'(x),F(z_I)-\Phi(z_I)\rangle|<\epsilon_1$ for all 
$x\in J_I'$, $I\in \Jcal$. Recall that $\bu_I$ and $\bv_I$ are orthogonal 
to $F(z_I)-\Phi(z_I)$.
\item[(IV)] $\Flux_{\wt F}=\Flux_{F'}$.
\end{enumerate}
Assuming that the numbers $\epsilon_0>0$ and $\epsilon_1>0$ 
are sufficiently small, $\wt F$ satisfies the conclusion of the lemma. 
Indeed, concerning condition (a), for every $x\in bX$ we have by (II), (vi), 
and \eqref{eq:distx0bK} that
\begin{equation}\label{eq:<<<}
	|\wt F(x)-F'(x)|<\eta+\epsilon_1,\quad |F'(x)-F(x)|
	<\epsilon_0,\quad\text{ and }\quad |F(x)-\Phi(x)|<\delta-\mu_0.
\end{equation}
On the other hand, \eqref{eq:Phie0} and (vi) give that
\begin{equation}\label{eq:almostorthogonal}
	|(F(x)-\Phi(x))-(F(z_I)-\Phi(z_I))|<2\epsilon_0\quad 
	\text{ for every }x\in I,\; I\in\Jcal.
\end{equation}
By the Schwarz inequality, this, (II), (III), and \eqref{eq:<<<} ensure 
for every $x\in J_I'$, $I\in\Jcal$, that
\begin{eqnarray*}
	|\langle \wt F(x)-F(x),F(x)-\Phi(x)\rangle| & \le &  
	|\langle F'(x)-F(x),F(x)-\Phi(x)\rangle| 
	\\ & &+ \; |\langle \wt F(x)-F'(x),F(z_I)-\Phi(z_I)\rangle| \\
	& & +\; |\langle \wt F(x)-F'(x),(F(x)-\Phi(x))-(F(z_I)-\Phi(z_I))\rangle|\\
	& < & \epsilon_0\delta + \epsilon_1+2\epsilon_0(\eta+\epsilon_1),
\end{eqnarray*}
thus $\wt F(x)-F(x)$ is approximately orthogonal to $F(x)-\Phi(x)$ when 
$\epsilon_0$ and $\epsilon_1$ are small. Hence, Pythagoras' theorem 
implies that
\[
	|\wt F(x)-\Phi(x)|< \sqrt{|\wt F(x)-F(x)|^2
	+|F(x)-\Phi(x)|^2}+O(\epsilon_0,\epsilon_1)
	\quad \text{ for every }x\in\Upsilon',
\]
where $O(\epsilon_0,\epsilon_1)>0$ is a constant depending on $\epsilon_0$
and $\epsilon_1$ that goes to $0$ as $\epsilon_0,\epsilon_1\to 0$. 
In view of \eqref{eq:<<<}, this guarantees that
\[
	|\wt F(x)-\Phi(x)|<\sqrt{(\eta+\epsilon_1+\epsilon_0)^2
	+(\delta-\mu_0)^2}+O(\epsilon_0,\epsilon_1)
	<\sqrt{\delta^2+\eta^2}\quad \text{ for every }x\in\Upsilon',
\]
where the latter inequality holds whenever $\epsilon_0>0$ and 
$\epsilon_1>0$ are sufficiently small.  Likewise, for every 
$x\in bX\setminus\Upsilon'$ we have by (II) and \eqref{eq:<<<} that
\[
	|\wt F(x)-\Phi(x)| \le |\wt F(x)-F'(x)|+|F'(x)-F(x)|+|F(x)-\Phi(x)|
	<\epsilon_1+\epsilon_0+\delta-\mu_0<\sqrt{\delta^2+\eta^2},
\]
where the last inequality holds for small enough $\epsilon_0$ and 
$\epsilon_1$. This proves (a). 

Regarding condition (b), consider a path $\gamma$ in 
$\overline X\setminus\mathring K$ connecting a point $x\in bK$ with a point 
$y\in bX$, and otherwise disjoint from $bK\cup bX$. 
If $y\in bX\setminus \mathring \Upsilon=\bigcup_{i=1}^m\Gamma_i$, 
there are $i\in\{1,\ldots,m\}$ and a subpath 
$\gamma'\subset\gamma\cap W_i$ connecting $bW_i\cap X$ with 
$\Gamma_i$. Taking into account \eqref{eq:almostorthogonal} and that the 
vectors $\bu_I$ and $\bv_I$ are orthogonal to $F(z_I)-\Phi(z_I)$, and 
assuming that the approximations in Theorem \ref{th:RH} to obtain 
$\wt F$ from $F'$ are close enough, we have that the projection of 
$\wt F|_{\gamma'}$ onto the vector $F(x_i)-\Phi(x_i)$ is very close to that 
of $F'|_{\gamma'}$ when $\epsilon_0>0$ is small. 
Note that the length of a curve cannot decrease very much under uniform 
approximation. Applying the argument of 
\cite[Claim 4.4]{AlarconDrinovecForstnericLopez2015PLMS} to the stronger 
projected-length condition imposed in part (vii) of the first deformation, we obtain 
\[ 
	\length(\widetilde F\circ\gamma)>\eta 
\] 
for every such path $\gamma$, provided that the approximation is sufficiently close.
If on the contrary $y\in \Upsilon\subset\Omega$, then in view of (I) and 
the first inequality in (II) there is a point $y'\in\gamma\cap X\cap b\Omega$ 
near $y$ such that
\[
	\length(\wt F(\gamma)) > |\wt F(y)-\wt F(y')| 
	\ge |\wt F(y)-F'(y)| -  |F'(y)-\wt F(y')|>\eta-\epsilon_1-\epsilon_0.
\]
Therefore, $\dist_{\wt F}(bK,bX)>\eta-\epsilon_1-\epsilon_0$.  
Choose the approximation on $K$ sufficiently close that 
\[ 
	\dist_{\widetilde F}(x_0,bK)>d+\frac{\mu_0}{2}
\] 
(see \eqref{eq:distx0bK}). Taking $\epsilon_0+\epsilon_1<\mu_0/2$, the preceding boundary-collar 
estimate gives $\dist_{\widetilde F}(x_0,bX)>d+\eta$, which is condition (b). 
Finally, (c) follows from (viii) and (IV).
\end{proof}

%
%
\subsection{A general position lemma}
As we have already mentioned, the proofs of Theorems \ref{th:proper} 
and \ref{th:complete-proper} require a recursive application of Lemma 
\ref{lem:proper}. For this lemma to be applied, we need to begin with a 
$G$-equivariant conformal minimal immersion $\overline X\to\R^n$ vanishing 
nowhere on $bX$; see \eqref{eq:lemproper}. Such an immersion is 
provided by the following result.

%
%
\begin{lemma}\label{lem:nozeros}
If $X$ and $G$ are as in Theorem \ref{th:proper}, then every 
$F\in\CMI_{G,*}^1(\overline X,\R^n)$ can be approximated as closely 
as desired uniformly on $\overline X$ by 
$\wt F\in\CMI_{G,*}^1(\overline X,\R^n)$ with $0\notin\wt F(bX)$.
\end{lemma}

%
%
\begin{proof}
By Remark \ref{rem:XG} and 
Theorem \ref{th:FFbis} we may assume that $X$ is a relatively compact 
$G$-invariant domain with real analytic boundary in an open Riemann 
surface $R$, that $G$ acts on $R$ by holomorphic automorphisms, 
that $G$ acts freely on $R\setminus X$, and that $F$ is a nondegenerate 
$G$-equivariant conformal minimal immersion $F:R\to\R^n$. Further, after 
shrinking $R$ around $\overline X$, we may assume that the stabiliser 
locus of the action of $G$ on $R$ equals the finite set 
\[ 
	X_0=\{x\in X:G_x\ne\{1\}\}\subset X. 
\] 
Since the action of \(G\) is free on a neighbourhood of \(bX\), 
the quotient projection $\pi:R\to R/G$ maps $\overline X$ properly
onto the smoothly bounded domain $\overline X/G\subset R/G$.
Since $F$ is an immersion on $\overline X$, its zero set is locally discrete.
Choose a neighbourhood $N$ of $bX$ such that
\[ 
	B:=F^{-1}(0)\cap N= F^{-1}(0)\cap bX;
\]
note that this is a finite set. Choose an open neighbourhood 
$U\subset R/G$ of $\overline X/G$ and a biholomorphism
\[
	\phi: U \to \phi(U) \subset R/G, 
\]
close to the identity on $\overline X/G$, such that $\phi$ is tangent to the 
identity at the points in $X_0/G$ and satisfies
\[
	\phi(bX/G) \subset (R/G)\setminus (B/G). 
\]
The last condition is achieved by a general position argument. 
For this, one uses the flow of a holomorphic vector field on the quotient, 
transverse to the boundary at $B/G$ and vanishing to sufficiently high order 
at the branch values $X_0/G$. Set $\wt U=\pi^{-1}(U)\subset R$. 
We see as in the proof of Theorem \ref{th:exposing} that $\phi$ lifts 
to a $G$-equivariant biholomorphism $\varphi:\wt U\to\varphi(\wt U)\subset R$ 
which approximates the identity on $\overline X$, it fixes the points of $X_0$
to a given finite order, and it satisfies
\[
	\varphi(bX)\subset N \qquad\text{and}\qquad 
	\varphi(bX)\cap B=\varnothing.
\]
The map $\wt F=F\circ\varphi : \wt U\to \R^n$ then satisfies the lemma.
\end{proof}

%
%
%
\section{Proof of the Main Theorems}\label{sec:proof}
%
%
\begin{proof}[Proof of Theorem \ref{th:complete-proper}]
In view of Remark \ref{rem:XG}, we regard $X$ as a relatively compact 
$G$-invariant domain with real analytic boundary in an open Riemann 
surface $R$ and assume that $G$ acts on $R$ by holomorphic 
automorphisms. After shrinking $R$ around $\overline X$, we may assume 
that the stabiliser locus of the action of $G$ on $R$ equals the finite set 
\begin{equation}\label{eq:X0-proof}
	X_0=\{x\in X:G_x\ne\{1\}\}\subset X.
\end{equation}
By Theorem \ref{th:FF} and Lemma \ref{lem:nozeros}, there is a 
nondegenerate $G$-equivariant conformal minimal immersion 
$F_0:R\to\R^n$ with $0\notin F_0(bX)$. We use the same notation for its 
restriction to $\overline X$.  Moreover, up to applying a homothety, 
we may assume that $F_0(\overline X)\subset \B$, where $\B$ denotes 
the unit ball in $\R^n$. In particular, there are numbers $0<r_0<r_1<1$ 
and a smoothly bounded compact 
$G$-invariant domain $K_0\subset X$ that is a strong 
deformation retract of $X$ such that
\begin{equation}\label{eq:F0-r0-r1} 
	F_0(\overline {X\setminus K_0})
	\subset r_1\B\setminus r_0\overline\B.
\end{equation}  
Choose any increasing sequence 
$r_1<r_2<r_3<\cdots$ such that
\begin{equation}\label{eq:rj-2}
	\lim_{j\to\infty} r_j=1 \quad\text{ and }
	\quad \sum_{j\ge 2} \sqrt{r_{j+1}^2-r_{j-1}^2}<\infty.
\end{equation}
Fix a point $x_0\in\mathring K_0$ and a number $\epsilon_0>0$. 
To complete the proof we shall inductively construct a sequence of 
nondegenerate $G$-equivariant conformal minimal immersions 
$F_j:\overline X\to\R^n$, smoothly bounded compact $G$-invariant 
domains $K_j\subset \mathring X$ that are strong deformation retracts 
of $X$, and positive numbers $\epsilon_j>0$, satisfying the following 
conditions for every $j=1,2,\ldots$
\begin{enumerate}[(a$_j$)] 
\item $K_{j-1}\subset \mathring K_j$.
\item $F_j(\overline{X\setminus K_j})\subset r_{j+1}\B\setminus r_j\overline\B$.
\item $F_j(\overline{X\setminus K_{j-1}})
	\subset r_{j+1}\B\setminus r_{j-1}\overline\B$.
\item $\|F_j-F_{j-1}\|_{1,K_{j-1}}<\epsilon_{j-1}$.
(Here, $\|\cdotp\|_{1,K}=\|\cdotp\|_{\Cscr^1(K)}$.)
\item $\|F_j-F_{j-1}\|_{0,\overline X}<\sqrt{r_{j+1}^2-r_{j-1}^2}$.
\item $\dist_{F_j}(x_0,bK_j)>j$.
\item $\epsilon_j<\epsilon_{j-1}/2$ and if $F:X\to\R^n$ is a conformal 
harmonic map such that 
$\|F-F_j\|_{1,K_j}<2\epsilon_j$, then 
$F|_{K_{j}}$ is a nondegenerate immersion, 
$F(\overline{K_j\setminus K_{j-1}})
 	\subset r_{j+1} \B\setminus r_{j-1}\overline\B$, 
and $\dist_{F}(x_0,bK_j)>j$.
\end{enumerate}
In addition, we will ensure that
\begin{equation}\label{eq:X=CupKj}
	X=\bigcup_{j\ge 1}K_j.
\end{equation}

With such sequences in hand, conditions (d$_j$), (g$_j$), and 
\eqref{eq:X=CupKj} ensure that there is a limit nondegenerate 
$G$-equivariant conformal minimal immersion
$F=\lim_{j\to\infty} F_j:X\to\R^n$ such that
\begin{equation}\label{eq:F-Fj-proof}
	\|F-F_j\|_{1,K_j}<2\epsilon_j\quad \text{ for every }j=0,1,2,\ldots
\end{equation}
By \eqref{eq:rj-2} and (e$_j$) the sequence $F_j$ also converges uniformly 
on $\overline X$ to a continuous $G$-equivariant map 
$F:\overline X\to \overline\B$, while  \eqref{eq:rj-2}, \eqref{eq:F-Fj-proof}, 
and (g$_j$) imply that the restriction $F|_X:X\to\B$ is proper and complete. 
Properness also gives $F(bX)\subset b\B$.

It therefore remains to explain the induction. The basis is given by the already 
fixed $F_0$, $K_0$, and $\epsilon_0$. Condition (b$_0$) holds by 
\eqref{eq:F0-r0-r1}, while (f$_0$) is implied by the fact that 
$x_0\in\mathring K_0$. All other conditions are vacuous for $j=0$. 
For the inductive step, assume that for some $j\ge 1$ we have $F_{j-1}$, 
$K_{j-1}$, and $\epsilon_{j-1}$ satisfying (b$_{j-1}$). This condition enables 
us to apply Lemma \ref{lem:proper} to the data
\[
	F=F_{j-1},\quad K=K_{j-1},\quad r=r_{j-1},\quad s=r_{j+1},
	\quad \delta=r_{j+1}-r_j,
\]
obtaining a nondegenerate $G$-equivariant conformal minimal immersion 
$F_j:\overline X\to\R^n$ satisfying (c$_j$), (d$_j$), (e$_j$), and 
\begin{equation}\label{eq:FjbX}
	F_j(bX)\subset r_{j+1}\B\setminus r_j\overline\B.
\end{equation}
To ensure that $F_j$ is nondegenerate, use that $F_{j-1}$ is nondegenerate 
and choose the approximation in (d$_j$) close enough.
In turn, by applying Lemma \ref{lem:equivariant-distance} to $F_j$ 
we can assume in addition that
\begin{equation}\label{eq:distFjbX}
	\dist_{F_j}(x_0,bX)>j.
\end{equation}
The additional perturbation is chosen sufficiently small uniformly on $\overline X$ 
and in the $\Cscr^1$ topology on $K_{j-1}$ so that conditions 
(c$_j$)--(e$_j$) and \eqref{eq:FjbX} remain valid.
By continuity of $F_j$ and compactness of $\overline X$, 
we can choose in view of \eqref{eq:FjbX} and \eqref{eq:distFjbX} a 
smoothly bounded compact $G$-invariant domain $K_j\subset X$ 
which is a strong deformation retract of $X$ and satisfies conditions
(a$_j$), (b$_j$), and (f$_j$). Finally, 
choose $\epsilon_j>0$ so small that (g$_j$) holds. 
It is clear that in this process we can choose the compact domain 
$K_j$ in each step large enough to ensure \eqref{eq:X=CupKj}. 
This closes the induction and completes the proof.
\end{proof}
%
%

%
%
\begin{proof}[Proof of Theorem \ref{th:proper}]
Reasoning as in the proof of Theorem \ref{th:complete-proper}, we can assume that we have objects $X$, $G$, $R$, $X_0$, $F_0$, $r_0$, and $K_0$ as in that proof such that \eqref{eq:X0-proof} holds and
\begin{equation}\label{eq:F0-r0-1} 
	F_0(\overline X\setminus\mathring K_0)
	\subset \B\setminus r_0\overline\B.
\end{equation} 
Cf.\ \eqref{eq:F0-r0-r1}. Fix $\epsilon_0>0$. A recursive application of Lemma \ref{lem:proper} then gives a sequence of nondegenerate $G$-equivariant conformal minimal immersions $F_j:\overline X\to\R^n$, smoothly bounded compact $G$-invariant domains $K_j\subset \mathring X$ that are strong deformation retracts 
of $X$, and positive numbers $\epsilon_j>0$, such that the following conditions are satisfied for every $j=1,2,\ldots$
\begin{enumerate}[(a$'_j$)] 
\item $K_{j-1}\subset \mathring K_j$.
\item $F_j(\overline{X\setminus K_j})\subset (j+1)\B\setminus j\overline\B$.
\item $F_j(\overline{X\setminus K_{j-1}})\subset (j+1)\B\setminus (j-1)\overline\B$.
\item $\|F_j-F_{j-1}\|_{1,K_{j-1}}<\epsilon_{j-1}$.
\item[(g$'_j$)] $\epsilon_j<\epsilon_{j-1}/2$ and if $F:X\to\R^n$ is a 
conformal harmonic map such that $\|F-F_j\|_{1,K_j}<2\epsilon_j$, then 
$F|_{K_j}$ is a nondegenerate immersion and 
$F(\overline{K_j\setminus K_{j-1}})\subset (j+1)\B\setminus (j-1)\overline\B$.
\end{enumerate}
In addition, we ensure \eqref{eq:X=CupKj}. This is done exactly as in the proof of 
Theorem \ref{th:complete-proper} but replacing $r_j$ by $j$ for $j=2,3,\ldots$; 
cf.\ conditions (a$_j$)--(d$_j$) and (g$_j$). 
For the initial step $j=1$, we apply Lemma \ref{lem:proper} with $r=r_0$, $s=2$, 
and $\delta=1$. (We may also guarantee for every $j\ge 2$ that 
$\|F_j-F_{j-1}\|_{0,\overline X}<2\sqrt{j}$ as in (e$_j$), 
but this is not needed for the current aim.) 
By (d$'_j$), \eqref{eq:X=CupKj}, and (g$'_j$), there is a limit nondegenerate $G$-equivariant conformal minimal immersion $F=\lim_{j\to\infty} F_j:X\to\R^n$ satisfying \eqref{eq:F-Fj-proof}. This condition and the second part of (g$'_j$) ensure that $F:X\to\R^n$ is a proper map.
\end{proof}
%
%

%
%
\begin{proof}[Proof of Theorem \ref{th:complete}]
As in the proof of Theorem \ref{th:complete-proper}, we may assume that 
$X$ is a relatively compact $G$-invariant domain with real analytic boundary 
in an open Riemann surface $R$, that $G$ acts on $R$ by holomorphic 
automorphisms, and that the stabiliser locus 
of the action of $G$ on $R$ equals the finite set $X_0$ in \eqref{eq:X0-proof}. 
For the existence assertion, Theorem \ref{th:FF} supplies an initial nondegenerate 
$G$-equivariant conformal minimal immersion on $R$.
For the approximation assertion, we start instead with the prescribed 
nonflat conformal $G$-equivariant minimal immersion $\overline X\to \R^n$ 
and apply Theorem \ref{th:FFbis} to approximate it uniformly on $\overline X$ 
by a nonflat $G$-equivariant conformal minimal immersion $F_0$ 
defined on a neighbourhood of $\overline X$.
Nondegeneracy is ensured by additional value interpolation at finitely many points outside $\overline X$ in the chosen neighbourhood. 

It remains to show that $F_0$ obtained above 
can be approximated as closely as desired uniformly 
on $\overline X$ by a continuous $G$-equivariant map $F: \overline X\to\R^n$ 
whose restriction $F|_{X}: X\to\R^n$ is a complete, nondegenerate, conformal 
minimal immersion.

Fix $\epsilon>0$. Choose a number 
$0<\epsilon_0<\epsilon/2$, a smoothly bounded compact 
$G$-invariant domain $K_0\subset X$ that is a strong deformation 
retract of $X$, and a point $x_0\in\mathring K_0$. 
To complete the proof, it suffices 
to construct a sequence of nondegenerate $G$-equivariant conformal 
minimal immersions $F_j:\overline X\to\R^n$, smoothly bounded compact 
$G$-invariant domains $K_j\subset \mathring X$ that are strong 
deformation retracts of $X$, and positive numbers $\epsilon_j>0$
satisfying \eqref{eq:X=CupKj} and the 
following conditions for every $j=1,2,\ldots$:
\begin{enumerate}[(A$_j$)] 
\item $K_{j-1}\subset \mathring K_j$.
\item $\|F_j-F_{j-1}\|_{0,\overline X}<\epsilon_{j-1}$.
\item $\|F_j-F_{j-1}\|_{1,K_{j-1}}<\epsilon_{j-1}$.
\item $\dist_{F_j}(x_0,bK_j)>j$.
\item $\epsilon_j<\epsilon_{j-1}/2$, and if $F:X\to\R^n$ is a 
conformal harmonic map such that $\|F-F_j\|_{1,K_j}<2\epsilon_j$
then $F|_{K_j}$ is a nondegenerate 
immersion and $\dist_{F}(x_0,bK_j)>j$.
\end{enumerate}
Indeed, by (B$_j$) and (E$_j$) there is a limit continuous $G$-equivariant map 
$
	F=\lim_{j\to\infty} F_j:\overline X\to\R^n
$
 with
\begin{equation}\label{eq:D-Fj-proof2}
	\|F-F_j\|_{0,\overline X}<2\epsilon_j\quad 
	\text{ for every }j=0,1,2,\ldots
\end{equation}
(Cf.\ \eqref{eq:F-Fj-proof}). In particular, 
$\|F-F_0\|_{0,\overline X}<2\epsilon_0<\epsilon$. 
Moreover, \eqref{eq:D-Fj-proof2} together with (C$_j$), (D$_j$) and (E$_j$) imply 
that $F|_X:X\to\R^n$ is a nondegenerate, complete, conformal minimal 
immersion with $\|F-F_j\|_{1,K_j}<2\epsilon_j$ for $j=0,1,2,\ldots$. 

The induction is done as follows. For the base case, note that $F_0$, $K_0$, 
and $\epsilon_0$ meet (D$_0$) while all other conditions are vacuous. 
Fix $j\ge 1$ and assume that we have objects $F_{j-1}$, $K_{j-1}$, 
and $\epsilon_{j-1}$ as above. (It is not needed that (D$_{j-1}$) holds.) 
Lemma \ref{lem:equivariant-distance} then provides a nondegenerate 
$G$-equivariant conformal minimal immersion $F_j:\overline X\to\R^n$ 
satisfying (B$_j$), (C$_j$) and $\dist_{F_j}(x_0,bX)>j$. 
(To ensure that $F_j$ is nondegenerate, use that $F_{j-1}$ is such 
and choose the approximation in (B$_j$) close enough.) 
This enables us to choose first $K_j$ so large that (A$_j$) and (D$_j$) 
hold (and to ensure \eqref{eq:X=CupKj} in the process), and then
$\epsilon_j>0$ so small that (E$_j$) holds. This closes the induction 
and completes the proof.
\end{proof}

%
%
\begin{proof}[Proof of Theorem \ref{th:NC}] 
We use the null-curve version of Theorem \ref{th:RH} from 
Corollary \ref{cor:RH-null}. Fix a $G$-invariant Hermitian inner 
product on $\C^n$, for example 
\[ 
	\langle u,v\rangle_G = \sum_{g\in G}\langle gu,gv\rangle_{\mathrm{st}}, 
\] 
where $\langle \cdotp,\cdotp\rangle_{\mathrm{st}}$ denotes the 
standard Hermitian inner product. 
The action of $G$ is unitary with respect to this Hermitian inner product. Its 
associated norm and Riemannian metric are equivalent to the standard 
Euclidean ones. Consequently, properness, boundedness, and completeness 
do not depend on which of these metrics is used. For null curves one works 
directly with holomorphic null immersions rather than with their real parts. 
The exposing construction on the source is unchanged. 
The differential data on the attached discs 
and the period-dominating sprays take values in the 
punctured null quadric, and all local data are chosen on one representative 
of each $G$-orbit and transported by the complex orthogonal action of $G$. 
Using the $G$-invariant Hermitian norm in the boundary-lifting argument, 
the proof of Theorem \ref{th:proper} gives part {\rm (a)}. Using the 
Riemannian metric associated to the same Hermitian inner product, 
the distance-increasing construction and the proof of Theorem 
\ref{th:complete} give part {\rm (b)}. Equivalence with the standard 
Euclidean metric completes the proof.
\end{proof}

%
%
%
%
\subsection*{Acknowledgements}
Alarc\'on is partially supported by the State Research Agency (AEI) via the grant no.\ PID2023-150727NB-I00, funded by MICIU/AEI/10.13039/501100011033 and ERDF/EU, Spain.
Forstneri\v c is supported by the European Union
(ERC Advanced grant HPDR, 101053085) 
and grant P1-0291 from ARIS, Republic of Slovenia. 

%
%
%
\subsection*{Declaration on the use of AI} 
The mathematical ideas, statements, and proofs in this article 
are due to the authors. Generative AI tools (mainly, GPT 6.1 Sol)
were used to verify mathematical correctness and completeness, and
to assist with language editing. The authors take full responsibility 
for the correctness and presentation of the article.



\begin{thebibliography}{10}

\bibitem{Alarcon2024Yang}
A.~Alarc{\'o}n.
\newblock The {Yang} problem for complete bounded complex submanifolds: a
  survey.
\newblock In {\em Cutting-edge mathematics. RSME 2022, Ciudad Real, Spain,
  January 17--21, 2022}, pages 1--25. Cham: Springer, 2024.

\bibitem{AlarconDrinovecForstnericLopez2015PLMS}
A.~Alarc{\'o}n, B.~Drinovec~Drnov{\v{s}}ek, F.~Forstneri\v{c}, and F.~J.
  L{\'o}pez.
\newblock Every bordered {R}iemann surface is a complete conformal minimal
  surface bounded by {J}ordan curves.
\newblock {\em Proc. Lond. Math. Soc. (3)}, 111(4):851--886, 2015.

\bibitem{AlarconDrinovecForstnericLopez2019TAMS}
A.~Alarc{\'o}n, B.~Drinovec~Drnov\v{s}ek, F.~Forstneri\v{c}, and F.~J.
  L{\'o}pez.
\newblock Minimal surfaces in minimally convex domains.
\newblock {\em Trans. Amer. Math. Soc.}, 371(3):1735--1770, 2019.

\bibitem{AlarconForstneric2013MA}
A.~Alarc{\'o}n and F.~Forstneri\v{c}.
\newblock Every bordered {R}iemann surface is a complete proper curve in a
  ball.
\newblock {\em Math. Ann.}, 357(3):1049--1070, 2013.

\bibitem{AlarconForstneric2015MA}
A.~Alarc{\'o}n and F.~Forstneri\v{c}.
\newblock The {C}alabi--{Y}au problem, null curves, and {B}ryant surfaces.
\newblock {\em Math. Ann.}, 363(3-4):913--951, 2015.

\bibitem{AlarconForstnericLopez2016MZ}
A.~Alarc{\'o}n, F.~Forstneri\v{c}, and F.~J. L{\'o}pez.
\newblock Embedded minimal surfaces in {$\mathbb {R}^n$}.
\newblock {\em Math. Z.}, 283(1-2):1--24, 2016.

\bibitem{AlarconForstnericLopez2021}
A.~Alarc\'{o}n, F.~Forstneri\v{c}, and F.~J. L\'{o}pez.
\newblock {\em Minimal surfaces from a complex analytic viewpoint}.
\newblock Springer Monographs in Mathematics. Springer, Cham, 2021.

\bibitem{AlarconLopez2012JDG}
A.~Alarc{\'o}n and F.~J. L{\'o}pez.
\newblock Minimal surfaces in {$\mathbb R^3$} properly projecting into
  {$\mathbb R^2$}.
\newblock {\em J. Differential Geom.}, 90(3):351--381, 2012.

\bibitem{Cerne2004}
M.~{\v{C}}erne.
\newblock Nonlinear {Riemann}-{Hilbert} problem for bordered {Riemann}
  surfaces.
\newblock {\em Am. J. Math.}, 126(1):65--87, 2004.

\bibitem{DAngelo2021}
J.~P. D'Angelo.
\newblock {\em Rational sphere maps}, volume 341 of {\em Prog. Math.}
\newblock Cham: Birkh{\"a}user, 2021.

\bibitem{DrinovecForstneric2007DMJ}
B.~Drinovec~Drnov{\v{s}}ek and F.~Forstneri\v{c}.
\newblock Holomorphic curves in complex spaces.
\newblock {\em Duke Math. J.}, 139(2):203--253, 2007.

\bibitem{Forstneric2023Indag}
F.~Forstneri{\v{c}}.
\newblock Recent developments on {Oka} manifolds.
\newblock {\em Indag. Math., New Ser.}, 34(2):367--417, 2023.

\bibitem{Forstneric2024PLMS}
F.~Forstneri{\v{c}}.
\newblock Minimal surfaces with symmetries.
\newblock {\em Proc. Lond. Math. Soc. (3)}, 128(3):32, 2024.
\newblock Id/No e12590.

\bibitem{Forstneric2003FM}
F.~Forstneri\v{c}.
\newblock The {O}ka principle for multivalued sections of ramified mappings.
\newblock {\em Forum Math.}, 15(2):309--328, 2003.

\bibitem{Forstneric2017E}
F.~Forstneri\v{c}.
\newblock {\em Stein manifolds and holomorphic mappings (The homotopy principle
  in complex analysis)}, volume~56 of {\em Ergebnisse der Mathematik und ihrer
  Grenzgebiete. 3. Folge}.
\newblock Springer, Cham, second edition, 2017.

\bibitem{ForstnericWold2009}
F.~Forstneri\v{c} and E.~F. Wold.
\newblock Bordered {R}iemann surfaces in {${\mathbb C}^2$}.
\newblock {\em J. Math. Pures Appl. (9)}, 91(1):100--114, 2009.

\bibitem{Grauert1958MA}
H.~Grauert.
\newblock Analytische {F}aserungen \"uber holomorph-vollst{\"a}ndigen
  {R}{\"a}umen.
\newblock {\em Math. Ann.}, 135:263--273, 1958.

\bibitem{GunningNarasimhan1967}
R.~C. Gunning and R.~Narasimhan.
\newblock Immersion of open {R}iemann surfaces.
\newblock {\em Math. Ann.}, 174:103--108, 1967.

\bibitem{JorgeXavier1980AM}
L.~P. d.~M. Jorge and F.~Xavier.
\newblock A complete minimal surface in {${\mathbb R}^{3}$} between two
  parallel planes.
\newblock {\em Ann. of Math. (2)}, 112(1):203--206, 1980.

\bibitem{KutzschebauchLarussonSchwarz2022}
F.~Kutzschebauch, F.~L{\'a}russon, and G.~W. Schwarz.
\newblock Equivariant {Oka} theory: survey of recent progress.
\newblock {\em Complex Anal. Synerg.}, 8(3):13, 2022.
\newblock Id/No 15.

\bibitem{Maskit1968}
B.~Maskit.
\newblock The conformal group of a plane domain.
\newblock {\em Am. J. Math.}, 90:718--722, 1968.

\bibitem{Miranda1995}
R.~Miranda.
\newblock {\em Algebraic curves and {Riemann} surfaces}, volume~5 of {\em Grad.
  Stud. Math.}
\newblock Providence, RI: AMS, American Mathematical Society, 1995.

\bibitem{Morales2003GAFA}
S.~Morales.
\newblock On the existence of a proper minimal surface in {$\mathbb{R}^3$} with
  a conformal type of disk.
\newblock {\em Geom. Funct. Anal.}, 13(6):1281--1301, 2003.

\bibitem{Nadirashvili1996IM}
N.~Nadirashvili.
\newblock Hadamard's and {C}alabi-{Y}au's conjectures on negatively curved and
  minimal surfaces.
\newblock {\em Invent. Math.}, 126(3):457--465, 1996.

\bibitem{SchoenYau1997}
R.~Schoen and S.~T. Yau.
\newblock {\em Lectures on harmonic maps}.
\newblock Conference Proceedings and Lecture Notes in Geometry and Topology,
  II. International Press, Cambridge, MA, 1997.

\bibitem{Stout1965TAMS}
E.~L. Stout.
\newblock Bounded holomorphic functions on finite {R}iemann surfaces.
\newblock {\em Trans. Amer. Math. Soc.}, 120:255--285, 1965.

\end{thebibliography}


\end{document}